\documentclass[12pt]{amsart}
\usepackage{amssymb,amsmath,amsthm}
\usepackage{verbatim,enumerate}

\DeclareMathOperator{\iindex}{index}
\DeclareMathOperator{\sign}{sign}

\DeclareMathOperator{\rank}{rank}
\DeclareMathOperator{\Ran}{Ran}
\DeclareMathOperator{\Dom}{Dom}
\DeclareMathOperator{\Ker}{Ker}
\DeclareMathOperator{\Tr}{Tr}
\DeclareMathOperator*{\slim}{s-lim}

\renewcommand\Im{\hbox{{\rm Im}}\,}
\renewcommand\Re{\hbox{{\rm Re}}\,}
\newcommand{\abs}[1]{\lvert#1\rvert}
\newcommand{\Abs}[1]{\left\lvert#1\right\rvert}
\newcommand{\norm}[1]{\lVert#1\rVert}

\newcommand{\R}{{\mathbb R}}

\newcommand{\C}{{\mathbb C}}

\newcommand{\calH}{{\mathcal H}}
\newcommand{\calK}{{\mathcal K}}
\newcommand{\calF}{\mathcal{F}}
\newcommand{\calN}{\mathcal{N}}
\newcommand{\calT}{\mathcal{T}}
\newcommand{\calM}{\mathcal{M}}
\newcommand{\calA}{\mathcal{A}}
\newcommand{\calB}{\mathcal{B}}

\numberwithin{equation}{section}

\theoremstyle{plain}
\newtheorem{theorem}{\bf Theorem}[section]
\newtheorem{lemma}[theorem]{\bf Lemma}
\newtheorem{proposition}[theorem]{\bf Proposition}

\newtheorem*{definition*}{\bf Definition}

\theoremstyle{definition}

\theoremstyle{remark}
\newtheorem*{remark*}{\bf Remark}
\newtheorem*{remarks*}{\bf Remarks}
\newtheorem{remark}[theorem]{\bf Remark}

\renewcommand{\qed}{\vrule height7pt width5pt depth0pt}
\newcommand{\wt}{\widetilde}
\newcommand{\ess}{{\rm ess}}

\begin{document}

\title[Birman-Schwinger principle]{The Birman-Schwinger principle on the 
essential spectrum}
\author{Alexander Pushnitski}
\address{Department of Mathematics\\
King's College London\\ 
Strand, London WC2R  2LS, U.K.}
\curraddr{}
\email{alexander.pushnitski@kcl.ac.uk}
\thanks{}



\begin{abstract}
Let $H_0$ and $H$ be self-adjoint operators in a Hilbert space.
We consider the spectral projections of $H_0$ and $H$ corresponding
to a semi-infinite interval of the real line. 
We discuss the index of this pair of spectral projections  
and   prove an identity which extends the Birman-Schwinger principle onto
the essential spectrum. 
We also relate this index to the spectrum of the scattering 
matrix for the pair $H_0$, $H$. 
\end{abstract}

\maketitle

\section{Introduction}\label{sec.z}

For a self-adjoint operator $H$ in a Hilbert space we denote by 
$E(\Lambda; H)$ the spectral projection of $H$ associated with 
a Borel set $\Lambda\subset\R$
and let 
$$
N(\Lambda;H)=\rank E(\Lambda; H)\leq \infty.
$$ 
Let $H_0$ and $H$ be two self-adjoint operators in a Hilbert 
space $\calH$; we wish to compare the eigenvalue distribution
functions   of $H_0$ and $H$. 
If our Hilbert space is finite dimensional, then the difference
\begin{equation}
N((-\infty,\lambda);H_0)-N((-\infty,\lambda);H)
\label{a1} 
\end{equation}
describes the shifts 
of the eigenvalues of $H$ relatively to the eigenvalues of $H_0$. 
Below we discuss a certain analogue of 
this difference
in the infinite dimensional case.

Throughout this paper, we assume that 
\begin{equation}
\begin{split}
&\text{$H_0$ and $H$
are semi-bounded from below with the same form domain}
\\
&\text{and the operator $V=H-H_0$ is $H_0$-form compact.}
\end{split}
\label{a1a}
\end{equation} 
This, in particular, ensures that the essential spectra
of $H_0$ and $H$ coincide: $\sigma_\ess(H_0)=\sigma_\ess(H)$.
Under these assumptions,  the difference \eqref{a1} is  of course still 
well defined
for $\lambda<\inf\sigma_\ess(H_0)$. 
The difficulty arises when the interval $(-\infty,\lambda)$ 
contains points of the  essential spectrum; 
then \eqref{a1} formally gives $\infty-\infty$. 

In this paper, we discuss the function
\begin{equation}
\Xi(\lambda; H,H_0)
=
\iindex\bigl (E((-\infty,\lambda);H_0),E((-\infty,\lambda);H)\bigr),
\label{a2}
\end{equation}
where the r.h.s. is the Fredholm index of a pair of projections,
the notion which is recalled in Section~\ref{sec.a2} below. 
As it will be clear from the discussion in Section~\ref{sec.a}, 
for $\lambda<\inf\sigma_\ess(H_0)$ we have
\begin{equation}
\Xi(\lambda;H,H_0)=
N((-\infty,\lambda);H_0)-N((-\infty,\lambda);H)
\label{a2a}
\end{equation}
and thus the definition \eqref{a2} provides a natural
regularisation of the difference \eqref{a1}. 
For $\lambda\in\R\setminus\sigma_\ess(H_0)$,
the index function 
$\Xi(\lambda;H,H_0)$ defined by \eqref{a2} has appeared 
before in the literature in various guises; we briefly discuss this 
in Section~\ref{sec.a3}. 
However, to the best of our knowledge, the index $\Xi(\lambda;H,H_0)$
for $\lambda$ \emph{on the essential spectrum} of $H_0$ has not been 
studied before. The purpose of this paper is to present a step in this 
direction. 
Our main result (Theorem~\ref{thm.a1} below) 
is an explicit formula for $\Xi(\lambda; H,H_0)$, $\lambda\in\sigma_\ess(H_0)$,
in terms of the ``sandwiched resolvent'' of $H_0$. 
This formula can be interpreted as an extension of the 
Birman-Schwinger principle onto the essential spectrum. 

To give the general flavour of our main  result, let us assume that 
$V\leq0$ in the quadratic form sense and suppose that the limit
$$
T_0(\lambda+i0)
=
\lim_{\varepsilon\to+0}\abs{V}^{1/2}(H_0-\lambda-i\varepsilon)^{-1}\abs{V}^{1/2}
$$
exists in the operator norm. 
Then, denoting $\Re T_0=(T_0+T_0^*)/2$, under the 
appropriate assumptions we prove that
\begin{equation}
\Xi(\lambda;H,H_0)
=
-N((1,\infty);\Re T_0(\lambda+i0)), 
\quad V\leq 0,
\label{z1}
\end{equation}
as long as $1$ is not an eigenvalue of $\Re T_0(\lambda+i0)$. 
For $\lambda<\inf\sigma(H_0)$, by virtue of \eqref{a2a}  this formula
simplifies to 
\begin{equation}
N((-\infty,\lambda);H)=N((1,\infty);T_0(\lambda+i0)), 
\quad V\leq0,
\label{z2}
\end{equation}
which is the Birman-Schwinger principle in its usual form.

Next, in the scattering theory framework we point out the following 
connection between $\Xi(\lambda; H,H_0)$ and the spectrum of the
scattering matrix $S(\lambda)$ corresponding to the pair $H_0$, $H$.
Recall that since $S(\lambda)$ is a unitary operator, the eigenvalues of $S(\lambda)$
are located on the unit circle in $\C$. 
Suppose that $\lambda$ is monotonically increasing, moving through
an interval of the absolutely continuous spectrum of $H_0$. Then 
every time that an eigenvalue of $S(\lambda)$ of multiplicity $n$ 
crosses the point $-1$ on the unit circle, the index $\Xi(\lambda; H,H_0)$
acquires a jump of $+n$ or $-n$. The jump of $+n$ occurs if
the eigenvalue of $S(\lambda)$ crosses $-1$ by rotating in a clockwise
direction, and $-n$ corresponds to the anti-clockwise rotation. 
See Theorem~\ref{thm.d1}.

Let us describe the structure of the paper. 
In Sections~\ref{sec.a2} and \ref{sec.a3} we recall 
the definition of the index of a pair of projections 
and collect the basic properties of the index  function 
$\Xi(\lambda;H,H_0)$ for $\lambda\notin\sigma_\ess(H_0)$. 
In Sections~\ref{sec.a4a} and \ref{sec.a4}, we 
recall the Birman-Schwinger principle for $\lambda\notin\sigma_\ess(H_0)$  
and state it in terms of the index function $\Xi(\lambda;H,H_0)$. 
In Section~\ref{sec.a5} we state our main result: the extension
of the Birman-Schwinger principle to the case $\lambda\in\sigma_\ess(H_0)$. 
Application to the Schr\"odinger operator is discussed in  
Section~\ref{sec.a8}.
In Section~\ref{sec.d}, we discuss the 
connection between the index function $\Xi(\lambda;H,H_0)$
and the spectrum of the scattering matrix $S(\lambda)$. 
The proof of the main result is given in Sections~\ref{sec.b}--\ref{sec.c}.

\section{Main results}\label{sec.a}

\subsection{The index of a pair of projections}\label{sec.a2}
Let $P,Q$ be orthogonal projections in a Hilbert space. 
By using some simple algebra 
(see e.g. \cite[Theorem~4.2]{ASS}) it is not difficult to see 
that 
$\sigma(P-Q)\subset[-1,1]$ and 
\begin{equation}
\dim \Ker(P-Q-\lambda I)=\dim \Ker(P-Q+\lambda I), 
\quad \lambda\not=\pm 1;
\label{a3}
\end{equation}
the proof of this is based on the identity
$$
(P-Q)W=W(Q-P), \quad W=I-P-Q.
$$
A pair $P,Q$  is called Fredholm, if
\begin{equation}
\{1,-1\}\cap\sigma_{\rm ess}(P-Q)=\varnothing.
\label{a3a}
\end{equation}
In particular, if $P-Q$ is compact, then the pair $P$, $Q$ is Fredholm.
The index of a Fredholm pair is defined by the formula
\begin{equation}
\iindex(P,Q)=\dim\Ker(P-Q-I)-\dim\Ker(P-Q+I).
\label{a4}
\end{equation}
We note that $\iindex(P,Q)$ coincides with the Fredholm index
of the operator $QP$ viewed as a map from $\Ran P$ to $\Ran Q$, 
see \cite[Proposition~3.1]{ASS}.

If $P-Q$ is a trace class operator,  then 
\begin{equation}
\iindex(P,Q)=\Tr(P-Q),
\label{b0}
\end{equation}
since all the eigenvalues of $P-Q$ apart from $1$ and $-1$  in the series 
$\Tr(P-Q)=\sum_k\lambda_k(P-Q)$
cancel out by \eqref{a3}. 
In the simplest case of finite rank projections $P,Q$ we have
$$
\iindex(P,Q)=\rank P-\rank Q.
$$
\subsection{Definition and basic properties of $\Xi$ }\label{sec.a3}
Let us accept the following
\begin{definition*}
Let $H_0$ and $H$ be self-adjoint operators in a Hilbert space.
Suppose that $E((-\infty,\lambda); H)$, $E((-\infty,\lambda);H_0)$ 
is a Fredholm pair. Then we will say that the index $\Xi(\lambda; H,H_0)$
exists and define it by 
$$
\Xi(\lambda; H,H_0)
=
\iindex\bigl (E((-\infty,\lambda);H_0),E((-\infty,\lambda);H)\bigr).
$$
\end{definition*}

Note that
by this definition, $\Xi(\lambda;H,H_0)$ is integer valued.
We need a simple existence statement for $\Xi$:
\begin{proposition}\label{rmk.a1}
Assume \eqref{a1a}. Then for all 
$\lambda\in\R\setminus\sigma_\ess(H_0)$ the difference 
of projections
$E((-\infty,\lambda); H)-E((-\infty,\lambda);H_0)$ is compact
and therefore the index 
$\Xi(\lambda;H,H_0)$ exists. 
\end{proposition}
This proposition is almost obvious, but for the sake of completeness
we give the proof at the end of Section~\ref{sec.a4a}.

Below, assuming \eqref{a1a}, we briefly recall the basic properties of $\Xi(\lambda; H,H_0)$.
Most of these properties have appeared before in the 
literature in various guises (see e.g. \cite{Klaus, ADH,Hempel1,Hempel2,Sobolev,GMN,GM,BPR,ACDS,ACS,KMS,Hempel,Simon100DM})
and can be regarded as folklore; 
they were reviewed  and proven in a systematic fashion 
in \cite{Push3}. 

For any $\lambda\in\R$, the index $\Xi(\lambda;H,H_0)$ exists if and only 
if $\Xi(\lambda;H_0,H)$ exists and if both of these indices exist, we have
\begin{equation}
\Xi(\lambda;H,H_0)=-\Xi(\lambda;H_0,H).
\label{a5}
\end{equation}
If $[a,b]\cap\sigma_\ess(H_0)=\varnothing$, then 
\begin{equation}
\Xi(b;H,H_0)-\Xi(a;H,H_0)
=
N([a,b);H_0)-N([a,b);H).
\label{a5a}
\end{equation}
In particular, we get \eqref{a2a} for $\lambda<\inf\sigma_\ess(H_0)$.  
For any $\lambda\in\R$, if $\Xi(\lambda;H,H_0)$ exists then the
estimates
\begin{equation}
-\rank V_-\leq\Xi(\lambda;H,H_0)\leq \rank V_+,
\quad
V_\pm=\frac12(\abs{V}\pm V)
\label{a6a}
\end{equation}
hold true. In particular, 
$$
\pm V\geq 0\ \Longrightarrow \   \pm\Xi(\lambda;H,H_0)\geq0.
$$
The estimates \eqref{a6a} can be improved if $\lambda$ is not in the 
spectrum of $H_0$. Suppose that for some $a>0$, one has
$[\lambda-a,\lambda+a]\cap\sigma(H_0)=\varnothing$. Then 
\cite[Corollary 3.3]{Push3} one has
\begin{equation}
-N((-\infty,-a);V)\leq \Xi(\lambda;H,H_0)\leq N((a,\infty);V).
\label{a6b}
\end{equation}
Next, if $V$ is a trace class operator, then 
\begin{equation}
\Xi(\lambda; H,H_0)=\xi(\lambda; H,H_0), 
\quad \lambda\in\R\setminus\sigma_\ess(H_0), 
\label{a7}
\end{equation}
where $\xi(\lambda; H,H_0)$ is  M.~G.~Krein's  spectral shift function.
See e.g. \cite[Chapter~8]{Yafaev} for a survey of the spectral shift function theory. 
Note that \eqref{a7} is in general false for 
$\lambda\in\sigma_\ess(H_0)$, since $\Xi$ is integer
valued and $\xi$ is real valued. 

\begin{remark}\label{rmk.a2}
$\xi(\lambda;H,H_0)$ and $\Xi(\lambda;H,H_0)$ are, in fact, two different
regularisations of 
\begin{equation}
\Tr\bigl(E((-\infty,\lambda);H_0)-E((-\infty,\lambda);H))\bigr).
\label{a7a}
\end{equation}
By an example due to M.~G.~Krein \cite{Krein} 
(see also Section~\ref{sec.a5a} below), 
the difference of spectral projections in \eqref{a7a} may 
fail to belong to the trace class if $\lambda\in\sigma_\ess(H_0)$. 
Thus, the trace in \eqref{a7a} may not exist. 
The spectral shift function is the regularisation of \eqref{a7a} 
obtained by replacing 
the difference of spectral projections 
by $\varphi(H)-\varphi(H_0)$, where $\varphi$ is a smooth 
approximation of the characteristic function of 
$(-\infty,\lambda)$.  
The index $\Xi(\lambda;H,H_0)$ is obtained by replacing 
$\Tr$ by $\iindex$ in \eqref{a7a}. These two regularisations 
coincide in simplest cases but in general are distinct.
\end{remark}

Finally, for $\lambda\in\R\setminus\sigma_\ess(H_0)$, 
the index $\Xi(\lambda;H,H_0)$ coincides with the spectral flow 
(i.e. the net flux of eigenvalues) of the operator family
$\{H_0+\alpha V\}_{\alpha\in[0,1]}$ through $\lambda$
as $\alpha$ increases monotonically from $0$ to $1$; 
see e.g. \cite[Section~2.6]{Push3}.
The spectral flow is particularly easy to define when 
$V\geq0$ or $V\leq0$; in this case the eigenvalues of 
$H_0+\alpha V$ are monotone in $\alpha$ and the spectral
flow is simply the total number of eigenvalues that 
cross the point $\lambda$ as $\alpha$ increases 
from $0$ to $1$. In general, one has to count
the eigenvalues with the sign plus or minus depending 
on whether they cross $\lambda$ to the right or to the left. 
See \cite{Hempel} for a comprehensive survey of the 
spectral flow in perturbation theory. 
We will return to the subject of spectral flow in Section~\ref{sec.d}
in the context of unitary operators. 

\subsection{The sandwiched resolvents and the resolvent identities}\label{sec.a4a}
The Birman-Schwinger principle is most conveniently stated
if the perturbation $V$ is factorised. 
Let us assume that $V$ is represented as $V=G^*JG$, where
$G$ is an operator from $\calH$ to an auxiliary Hilbert
space $\calK$ and $J$ is an operator
in $\calK$. We assume that  
\begin{equation}
\begin{split}
&\text{$J=J^*$,  $J$ is bounded and has a bounded inverse, }
\\
&\text{$\Dom(H_0-aI)^{1/2}\subset\Dom G$  \  and \   
$G(H_0-aI)^{-1/2}$ is compact, $\forall a<\inf\sigma(H_0)$.}
\end{split}
\label{a8}
\end{equation}
These assumptions ensure 
(by the ``KLMN Theorem'', see e.g.  \cite[Theorem~X.17]{RS2})
that $V$ is $H_0$-form compact and $H$  
coincides with the form sum $H_0+V$. 
Thus, \eqref{a1a} follows from \eqref{a8}. 
In fact, \eqref{a8} is just another way of stating the 
assumption \eqref{a1a}. Indeed, assuming \eqref{a1a}, 
one can always take 
$\calK=\calH$, $G=\abs{V}^{1/2}$ 
and\footnote{Here and in what follows $\sign(x)=1$ for $x\geq0$ and $\sign(x)=-1$ for $x<0$. 
In particular, $\sign(V)$ has a bounded inverse.}
 $J=\sign(V)$ and then \eqref{a8} holds true.
In applications, the factorisation $V=G^*JG$ often arises naturally
due to the structure of the problem.

Note that since $H_0$ and $H$ have the same form domain, 
under the assumption \eqref{a8} we also have
\begin{equation}
\text{$\Dom(H-aI)^{1/2}\subset\Dom G$ \  and \ 
$G(H-aI)^{-1/2}$ is compact}
\label{a9}
\end{equation}
for any $a<\inf\sigma(H)$.

For $z\in\C\setminus\sigma(H_0)$, let us denote the resolvent of $H_0$ by 
$R_0(z)=(H_0-zI)^{-1}$; similarly, let  $R(z)=(H-zI)^{-1}$ 
for $z\in\C\setminus\sigma(H)$.
Let us
define the operators $T_0(z)$, $T(z)$ (sandwiched resolvents) formally 
by setting
$$
T_0(z)=GR_0(z)G^*, 
\quad 
T(z)=GR(z)G^*.
$$
More precisely, this means
\begin{align}
T_0(z)=G(H_0-aI)^{-1/2}(H_0-aI)R_0(z)(G(H_0-aI)^{-1/2})^*,
\quad a<\inf\sigma(H_0),
\label{a10}
\\
T(z)=G(H-aI)^{-1/2}(H-aI)R(z)(G(H-aI)^{-1/2})^*,
\quad a<\inf\sigma(H).
\label{a11}
\end{align}
By \eqref{a8}, \eqref{a9}, the operators
$T_0(z)$, $T(z)$ are compact. 
The operator $T_0(z)$ is self-adjoint for all  
$z\in\R\setminus\sigma(H_0)$ and 
$T(z)$ is self-adjoint for all $z\in\R\setminus\sigma(H)$.

For future reference, let us display the iterated resolvent identity for the 
operators $H_0$ and $H$: 
\begin{equation}
R(z)-R_0(z)
=
-(GR_0(\overline{z}))^*J(GR(z))
=-(GR_0(\overline{z}))^*(J-JT(z)J)(GR_0(z))
\label{a10a}
\end{equation}
and its direct consequence
\begin{equation}
(J^{-1}+T_0(z))(J-JT(z)J)
=
(J-JT(z)J)(J^{-1}+T_0(z))
=
I.
\label{a10b}
\end{equation}
From \eqref{a10a}, in particular, we easily obtain
\begin{proof}[Proof of Proposition~\ref{rmk.a1}]
Let $\Gamma$ be a compact positively oriented contour in $\C\setminus(\sigma(H_0)\cup\sigma(H))$
such that the bounded set $(\sigma(H)\cup\sigma(H_0))\cap(-\infty,\lambda)$ 
is contained inside $\Gamma$. 
Then 
$$
E((-\infty, \lambda);H)
-
E((-\infty, \lambda);H_0)
=
\frac1{2\pi i}\int_\Gamma ((R_0(z)-R(z))dz.
$$
From \eqref{a10a} and \eqref{a8}, \eqref{a9} it is easy to see that the operator 
in the r.h.s. is compact, as required.
\end{proof}

\subsection{The Birman-Schwinger principle}\label{sec.a4}
In what follows, we assume \eqref{a8}.
We first note that by Proposition~\ref{rmk.a1}, 
for all $\lambda\in\R\setminus\sigma(H_0)$ the indices
$\Xi(\lambda;H,H_0)$ and   
$\Xi(0;J^{-1}+T_0(\lambda),J^{-1})$ 
exist. 
\begin{proposition}\label{prp.a1}
Assume \eqref{a8}. Then 
\begin{align}
\dim\Ker(H-\lambda I)
&=
\dim\Ker(J^{-1}+T_0(\lambda)),
\qquad \forall \lambda\in\R\setminus\sigma(H_0),
\label{a11a}
\\
\Xi(\lambda;H,H_0)
&=
-\Xi(0;J^{-1}+T_0(\lambda); J^{-1}),
\quad \forall \lambda\in\R\setminus(\sigma(H_0)\cup\sigma(H)).
\label{a12}
\end{align}
In particular, in  the cases $J=I$ or $J=-I$, the identity \eqref{a12} can be written as
\begin{align}
\Xi(\lambda;H_0+G^*G,H_0)
&=
N((-\infty, -1); T_0(\lambda)),
\label{a12a}
\\
\Xi(\lambda;H_0-G^*G,H_0)
&=
-N((1,\infty); T_0(\lambda)).
\label{a12b}
\end{align}
\end{proposition}
Note that for $\lambda<\inf\sigma(H_0)$, formula
\eqref{a12b} is equivalent to \eqref{z2}. 

Formula \eqref{a12} has a long history starting from the celebrated 
papers by M.~Sh.~Birman~\cite{Birman} and J.~Schwinger
\cite{Schwinger} where it was stated in the form equivalent to \eqref{z2}. 
The identities \eqref{a12a}, \eqref{a12b} were extensively used 
(see e.g. \cite{Klaus,DH,Hempel1,ADH,Hempel2}) in 
the context of the spectral flow and also in 
\cite[Theorem~3.5]{Sobolev}
in the context of the spectral shift function theory (see \eqref{a7}).
The identity \eqref{a12}
as stated above, i.e. in terms of the index of a pair of 
projections, was proven in \cite{GM} 
in the context of the spectral shift function theory 
for trace class perturbations $V$. 
It was extended to the general case in \cite{Push3}.
\begin{remark*}
The right hand side of \eqref{a12} is not symmetric with respect
to the interchange of $H_0$ and $H$. However, 
under the assumptions of Proposition~\ref{prp.a1}
by writing $H=H_0-V$ and using \eqref{a5}, one also obtains 
$$
\Xi(\lambda;H,H_0)=\Xi(0;J^{-1}-T(\lambda); J^{-1}),
\quad \forall \lambda\in\R\setminus(\sigma(H_0)\cup\sigma(H)).
$$
\end{remark*}
Our main result
below is an extension of \eqref{a12}  to the case
when $\lambda$ belongs to the essential spectrum of $H_0$.

\subsection{Main result}\label{sec.a5}
As above, we assume that the perturbation $V$ is factorised 
as  $V=G^*JG$ 
with the properties \eqref{a8} and use the notation $T_0(z)$
for the sandwiched resolvent.   
Let $\Delta\subset\R$ 
be an open interval. Assume that 
\begin{equation}
\begin{split}
&\text{$T_0(z)$ is uniformly continuous in the operator norm}
\\
&\qquad \text{in the rectangle  $\Re z\in\Delta$, $\Im z\in(0,1)$.}
\end{split}
\label{a15}
\end{equation}
Of course, this trivially implies that the limit
$T_0(\lambda+i0)$
exists in the operator norm and is continuous in $\lambda\in\Delta$.
The operator $T_0(\lambda+i0)$ is compact and in general 
non-selfadjoint.  We denote
\begin{equation}
A_0(\lambda)=\Re T_0(\lambda+i0),
\quad
B_0(\lambda)=\Im T_0(\lambda+i0),
\label{a15a}
\end{equation}
where $\Re X=(X+X^*)/2$,
$\Im X=(X-X^*)/2i$. 
 We also set 
\begin{equation}
\calN=\{\lambda\in\Delta\mid 0\in\sigma(J^{-1}+A_0(\lambda))\}.
\label{a18}
\end{equation} 
Below is our main result. For the purposes of future 
reference, we break up the statement of this 
theorem into several parts. 
\begin{theorem}\label{thm.a1}
Assume \eqref{a8} and \eqref{a15}.
Then: 
\begin{enumerate}[\rm (i)]
\item
the set $\calN$ defined by \eqref{a18} is closed in $\Delta$
(i.e. $\Delta\setminus\calN$ is open); 
\item
for all $\lambda\in\Delta\setminus\calN$, 
the index $\Xi(\lambda;H,H_0)$ exists; 
\item
for all $\lambda\in\Delta\setminus\calN$, 
 the identity 
\begin{equation}
\Xi(\lambda;H,H_0)=-\Xi(0;J^{-1}+A_0(\lambda);J^{-1})
\label{a17}
\end{equation}
holds true;
\item
the index $\Xi(\lambda;H,H_0)$  is constant on every
connected component of the set $\Delta\setminus\calN$.
\end{enumerate}
\end{theorem}
The proof is given in Sections~\ref{sec.b}--\ref{sec.c}. 
The proof uses Proposition~\ref{prp.a1} and a certain 
continuous deformation argument. Roughly speaking, 
we reduce Theorem~\ref{thm.a1} to Proposition~\ref{prp.a1}
by making an ``infinitesimal spectral gap'' in the 
spectrum of $H_0$ near $\lambda$. 

\begin{remarks*}
\begin{enumerate}[1.]
\item
The most important statement in Theorem~\ref{thm.a1} is part (iii). 
Part (i) is trivial, part (ii) follows from the results of \cite{Push4},
and part  (iv) is an easy consequence of part (iii).
\item
The existence of $\Xi(0;J^{-1}+A_0(\lambda),J^{-1})$ in the r.h.s. of 
\eqref{a17} follows from 
Proposition~\ref{rmk.a1} and from the fact that 
$A_0(\lambda)$ is compact.
\item
If $\lambda\in\R\setminus\sigma(H_0)$, 
then the hypothesis of Theorem~\ref{thm.a1} is trivially satisfied
(with $\Delta$ being a sufficiently small neighbourhood of $\lambda$)
and $T_0(\lambda+i0)$ is self-adjoint. 
Thus, in this case \eqref{a17} coincides with \eqref{a12}. 
\item
If $J=I$ or $J=-I$, then \eqref{a17} becomes
\begin{align*}
\Xi(\lambda;H_0+G^*G,H_0)&=N((-\infty,-1);A_0(\lambda)),
\\
\Xi(\lambda;H_0-G^*G,H_0)&=-N((1,\infty);A_0(\lambda)).
\end{align*}
In particular, we obtain \eqref{z1}. 
\item
Let $\Delta\subset\R\setminus\sigma(H_0)$. 
Then, by \eqref{a11a}, $\calN=\sigma(H)\cap\Delta$. 
Equivalently, $\calN$ is the set of all discontinuities (jumps) 
of $\Xi(\lambda;H,H_0)$ on $\Delta$. 
\end{enumerate}
\end{remarks*}
According to \eqref{a5a},  away from $\sigma_\ess(H_0)$
the jumps of the function $\Xi(\lambda;H,H_0)$ 
occur at the eigenvalues of $H_0$ and $H$.
Thus, one is tempted to interprete the jumps of $\Xi(\lambda;H,H_0)$ on the 
essential spectrum as  certain ``pseudo-eigenvalues'' of $H_0$ or $H$, 
depending on the sign of the jump. 
In the framework of Theorem~\ref{thm.a1} we see that these ``pseudo-eigenvalues''
can occur only at the points of the set $\calN$. 
In Section~\ref{sec.d}, we give an alternative description of these ``pseudo-eigenvalues''
in terms of the scattering matrix $S(\lambda)$ for the pair $H_0$, $H$.

\subsection{The set $\calN$: example}\label{sec.a5a}
The following example shows that the set $\calN$ can be 
quite large:  $\calN=\Delta$. 
In \cite{Krein}, M.~G.~Krein considered the operator $H_0$ 
in $L^2(0,\infty)$ with the integral kernel $H_0(x,y)$ given by 
$$
H_0(x,y)=
\begin{cases}
\sinh(x) e^{-y}, \quad x\leq y,
\\
\sinh(y) e^{-x}, \quad x\geq y
\end{cases}
$$
and the operator $H$ in the same Hilbert space with the integral 
kernel $H(x,y)=H_0(x,y)+e^{-x}e^{-y}$. 
Thus, $V=H-H_0$ is a rank one operator. 
In fact, $H_0=(h_0+I)^{-1}$ and $H=(h+I)^{-1}$, where
$h_0$ (resp. $h$) is the self-adjoint realisation of 
the operator
$-\frac{d^2}{dx^2}$ in $L^2(0,\infty)$ with the Dirichlet (resp. Neumann)
boundary condition at zero. 
In this example, $\sigma(H_0)=\sigma(H)=[0,1]$. 

M.~G.~Krein showed that for any $\lambda\in(0,1)$, 
the difference
$$
E((-\infty,\lambda);H)-E((-\infty,\lambda);H_0)
$$
does not belong to the Hilbert-Schmidt class. 
The more detailed analysis of \cite{KM} shows that for any $\lambda\in(0,1)$, 
$$
\sigma_\ess\bigl(E((-\infty,\lambda);H)-E((-\infty,\lambda);H_0)\bigr)=[-1,1]
$$
and so $\Xi(\lambda;H,H_0)$ does not exist for any $\lambda\in(0,1)$. 

In this example, the rank one perturbation $V$ can be factorised as 
$V=G^*G$, with $G:L^2(0,\infty)\to\C$, 
$Gf=\int_0^\infty f(x)e^{-x}dx$. 
Thus, the operator $T_0(z)$ reduces to a multiplication by a scalar in $\C$. 
Using the explicit formula for the resolvent of $h_0$, one easily checks that
$$
T_0(\lambda+i0)=-1+i\sqrt{\lambda^{-1}-1},
\quad
A_0(\lambda)=-1, \quad \forall\lambda\in(0,1),
$$
and therefore $\calN=\Delta$.

Considering rank one perturbations, it is not difficult to construct 
examples when the set $\calN$ has a more complex structure. 
We shall not pursue this direction here. 
On the other hand, Theorem~\ref{thm.a5} in the next subsection shows that 
in some situations of applied  interest, 
the set $\calN$ consists of isolated points.

\subsection{Application: Schr\"odinger operator}\label{sec.a8}
Let $H_0=-\Delta$ in $\calH=L^2(\R^d)$ with $d\geq1$ and 
let $H=H_0+V$ where $V$ is the operator of multiplication 
by a function (potential in physical terminology) $V:\R^d\to\R$. 
We assume that $V$ is 
a short range potential, i.e. 
\begin{equation}
\abs{V(x)}\leq C(1+\abs{x})^{-\rho}, \quad \rho>1.
\label{a19}
\end{equation}
Let us discuss the index function $\Xi(\lambda;H,H_0)$. 
For $\lambda<0$, this function
reduces to the eigenvalue counting function, see
\eqref{a2a}.
In order to analyse the index function for $\lambda>0$, let us apply 
Theorem~\ref{thm.a1}. 
Let $\calK=\calH$, $G=\abs{V}^{1/2}$, $J=\sign V$. 
Under the assumption \eqref{a19}, 
the hypotheses \eqref{a8} and \eqref{a15} are satisfied
with $\Delta=(\lambda_1,\lambda_2)$ for any 
$0<\lambda_1<\lambda_2<\infty$; 
see e.g. \cite[Theorem~XIII.33]{RS3}.
Thus, for any $\lambda>0$ formula \eqref{a17} 
holds true. 
The operator $A_0(\lambda)$ in this case is the self-adjoint
integral operator in $L^2(\R^d)$ with the kernel 
\begin{equation}
\abs{V(x)}^{1/2}\abs{V(y)}^{1/2}
\frac14 (2\pi)^{-\nu}k^{d-2}\frac{J_\nu(k\abs{x-y})}{(k\abs{x-y})^\nu}, 
\quad x,y\in\R^d,
\label{c16} 
\end{equation}
where $\nu=(d-2)/2$, 
$k=\sqrt{\lambda}>0$,
and $J_\nu$ is the Bessel function.
We have
\begin{theorem}\label{thm.a6}
Assume \eqref{a19}. 
For any $\lambda>0$, if $\Xi(\lambda;H,H_0)$ exists
then it satisfies the estimates
\begin{equation}
-N([1,\infty);A_0(\lambda))
\leq
\Xi(\lambda;H,H_0)
\leq
N((-\infty,-1];A_0(\lambda)).
\label{c17}
\end{equation}
Moreover, for all sufficiently large $\lambda>0$ 
the index $\Xi(\lambda;H,H_0)$ exists and equals zero. 
\end{theorem}
\begin{proof}
Since $\sigma(J^{-1})=\{-1,1\}$, we can apply 
\eqref{a6b} to the r.h.s. of \eqref{a17} with any $a\in(0,1)$. 
This yields
$$
-N((a,\infty);A_0(\lambda))
\leq
\Xi(\lambda;H,H_0)
\leq
N((-\infty,-a);A_0(\lambda)).
$$
Taking $a\to1$, we obtain \eqref{c17}.

Next, 
under the assumption \eqref{a19}, one has
(see e.g. \cite[Problem 60, page 390]{RS3}):
\begin{equation}
\norm{T_0(\lambda+i0)}\to0 \quad \text{ as $\lambda\to+\infty$.}
\label{a20}
\end{equation}
Thus, for all sufficiently large $\lambda>0$ one has $\norm{A_0(\lambda)}<1$. 
For such $\lambda$, the operator $J^{-1}+A_0(\lambda)=J^{-1}(I+JA_0(\lambda))$ is invertible. 
Thus by Theorem~\ref{thm.a1}(ii) the index $\Xi(\lambda;H,H_0)$ exists. 
For such $\lambda$ we have
$$
N((-\infty,-1];A_0(\lambda))
=
N([1,\infty);A_0(\lambda))
=0
$$
and therefore by \eqref{c17} we get $\Xi(\lambda;H,H_0)=0$, as required.
\end{proof}
Theorem~\ref{thm.a6} can be combined with spectral estimates for $A_0(\lambda)$ 
to yield explicit bounds for $\Xi(\lambda;H,H_0)$ in terms of $V$. 
Let us give a simple example of such a bound. 
Let $d=3$. Then the integral kernel of $A_0(\lambda)$, $\lambda=k^2>0$, is 
$$
\abs{V(x)}^{1/2}\abs{V(y)}^{1/2}
\frac{\cos k\abs{x-y}}{4\pi \abs{x-y}}.
$$
Using the estimate
$$
N([1,\infty);\pm A_0(\lambda))\leq \norm{A_0(\lambda)}_2^2
$$
in terms of the Hilbert-Schmidt norm $\norm{\cdot}_2$, 
we obtain
$$
\abs{\Xi(\lambda;H,H_0)}
\leq 
\frac{1}{16\pi^2}\int_\R \int_\R 
\frac{\abs{V(x)}\abs{V(y)}}{\abs{x-y}^2}dx\, dy,
$$
whenever the integral in the r.h.s. converges.

Under additional assumptions on the potential $V$, one can ensure 
that the set $\calN$ is finite:
\begin{theorem}\label{thm.a5}
Assume that $\abs{V(x)}\leq \exp(-\gamma\abs{x})$ with some $\gamma>0$. 
Then the index $\Xi(\lambda;H,H_0)$ exists for all $\lambda\in\R\setminus \calN_0$,
where $\calN_0$ is a finite set.  
\end{theorem}
\begin{proof}
By Proposition~\ref{rmk.a1}, the index $\Xi(\lambda;H,H_0)$ 
exists for all $\lambda<0$. By Theorem~\ref{thm.a1}, it suffices to prove
that $I+JA_0(\lambda)$ is invertible for all $\lambda>0$ apart from 
a finite set. 
Let us use formula \eqref{c16}. 
It is well known that $z^{-\nu}J_\nu(z)$ is an entire function of $z$ which 
obeys 
$$
\Abs{z^{-\nu}J_\nu(z)}\leq 
\frac{\exp({\abs{\Im z}})}{2^\nu\Gamma(\nu+1)},
\quad
\nu\geq -1/2.
$$
It follows that the operator $A_0(k^2)$ 
is analytic in $k$ for $\abs{\Im k}<\gamma/2$ and $d\geq2$. 
For $d=1$, the operator $A_0(k^2)$ is analytic in $k$ for 
$\abs{\Im k}<\gamma/2$, $k\not=0$ and has a single pole at $k=0$.
By \eqref{a20}, the operator $I+JA_0(\lambda)$ is invertible for all 
sufficiently large $\lambda$. 
By the analytic Fredholm alternative,
we see that $I+JA_0(\lambda)$ is invertible for all but finitely many 
$\lambda>0$. 
\end{proof}

\section{$\Xi$ and the scattering matrix}\label{sec.d}

Below we recall the definition of the scattering matrix $S(\lambda)$ for the pair 
$H_0$, $H$ and define the \emph{spectral flow}  $\mu(e^{i\theta};\lambda)$  
of the scattering matrix. 
Next, we establish a formula \eqref{d8} which relates 
$\Xi(\lambda;H,H_0)$ 
and the spectral flow.
This formula allows one to describe  the jumps of $\Xi(\lambda;H,H_0)$ 
in terms of the spectrum of the scattering matrix. 

For the purposes of simplicity and clarity, we restrict the discussion in this
section to the case of the Schr\"odinger operator. 
However, the construction of this section can be extended to a much wider 
setting, see Remark~\ref{rmk.d2}.
The proof of Theorem~\ref{thm.a1} does not use the material of this section.

\subsection{The spectral flow for unitary operators}\label{sec.d0}
We start by defining the spectral flow of a family of unitary operators in an 
abstract setting. 
Let $U=U(t)$, $t\in[a,b]$, be a family of unitary operators in a Hilbert space such that 
$U(t)$ depends continuously on $t$ in the operator norm and such 
that $U(t)-I$ is compact for all $t$. 
Since $U(t)$ is unitary, the spectrum of $U(t)$ is a subset of the unit circle 
$\mathbb T$. Since $U(t)-I$ is compact, the spectrum of 
$U(t)$ away from $1$ consists of eigenvalues of finite multiplicities;  the 
only possible point of  accumulation of these eigenvalues is $1$. 

Let us recall the definition of the spectral flow of the family $\{U(t)\}_{t\in[a,b]}$. 
The spectral flow is an integer valued function $\mu$ 
on $\mathbb T\setminus\{1\}$. The naive definition of the spectral flow is 
\begin{multline}
\mu(e^{i\theta}; \{U(t)\}_{t\in[a,b]})=
\\
\langle \text{the number of eigenvalues of $U(t)$ which cross
$e^{i\theta}$ in the anti-clockwise direction}\rangle
\\
-
\langle \text{the number of eigenvalues of $U(t)$ which cross
$e^{i\theta}$ in the clockwise direction}\rangle,
\label{d4}
\end{multline}
as $t$ increases monotonically from $a$ to $b$. 
Here $\theta\in(0,2\pi)$ and the eigenvalues
are counted with multiplicities taken into account. 
The eigenvalues of $U(t)$ may cross $e^{i\theta}$ infinitely many times, 
and thus the above naive definition needs to be replaced by a more robust one.
Below we describe one of such possible regularisations. 
 
Let us introduce some notation for the eigenvalue counting function of a 
unitary operator.  
For $\theta_1,\theta_2\in(0,2\pi)$ denote 
$$
N(e^{i\theta_1},e^{i\theta_2}; U(t))
=
\sum_{\theta\in[\theta_1,\theta_2)}
\dim\Ker(U(t)-e^{i\theta}I)
$$
if $\theta_1<\theta_2$ and 
$$
N(e^{i\theta_1},e^{i\theta_2}; U(t))
=
-N(e^{i\theta_2},e^{i\theta_1}; U(t))
$$
if $\theta_1>\theta_2$. Assume first that there exists $\theta_0\in(0,2\pi)$ 
such that $e^{i\theta_0}\notin\sigma(U(t))$ for all $t\in[a,b]$. 
Then one can define the spectral flow of the family $\{U(t)\}_{t\in[a,b]}$ 
by 
\begin{equation}
\mu(e^{i\theta};\{U(t)\}_{t\in[a,b]})
=
N(e^{i\theta},e^{i\theta_0};U(b))
-
N(e^{i\theta},e^{i\theta_0};U(a)).
\label{d5}
\end{equation}
It is evident that this definition is independent of the choice of $\theta_0$ and 
agrees with the naive definition \eqref{d4} whenever the latter makes sense.

In general, $\theta_0$ as above may not exist. 
However, by a compactness argument one can always find 
the values $a=t_0<t_1<\dots<t_n=b$ 
such that 
for each of the subintervals $\Delta_i=[t_{i-1},t_i]$, 
a point $\theta_0$  with the required properties
can be found. 
Thus, the spectral flow of each of the corresponding families
$\{U(t)\}_{t\in\Delta_i}$ is well defined. Now one can set
\begin{equation}
\mu(e^{i\theta};\{U(t)\}_{t\in[a,b]})
=
\sum_{i=1}^n
\mu(e^{i\theta}; \{U(t)\}_{t\in\Delta_i}).
\label{d5a}
\end{equation}
It is not difficult to see that the above definition is independent 
on the choice of the subintervals $\Delta_i$ and agrees with the naive 
definition \eqref{d4}.

\subsection{The scattering matrix}\label{sec.d1}
Throughout the rest of this section, we assume that 
$\calH=L^2(\R^d)$ and let $H_0=-\Delta$ and $H=H_0+V$ be as in Section~\ref{sec.a8}, 
where $V$ satisfies the short range assumption \eqref{a19}.
Let us recall the definition of the scattering matrix 
$S(\lambda)$ for the pair $H_0$, $H$; see e.g. \cite{Yafaev}.
If the potential $V$ is short range \eqref{a19}, then the wave operators 
$$
W_\pm=
\slim_{t\to\pm\infty}
e^{itH}e^{-itH_0}
$$
exist and are asymptotically complete. This means that
the singular continuous spectrum of $H$ is absent and 
$\Ran W_+=\Ran W_-=\calH_{pp}(H)^\bot$, 
where  $\calH_{pp}(H)\subset\calH$ is the subspace 
spanned by the eigenfunctions of $H$. 
The scattering operator $\mathbf S=W_+^*W_-$ 
is unitary in $\calH$ and commutes
with $H_0$.

Consider the map $\calF:L^2(\R^d)\to L^2((0,\infty);L^2(\mathbb S^{d-1}))$
(here  $\mathbb S^0=\{-1,1\}$), which for $f\in L^1(\R^d)$ is defined by 
$$
(\mathcal F f)(\lambda;\omega)
=
2^{-1/2}\lambda^{(d-2)/4}(2\pi)^{-d/2}
\int_{\R^d} f(x) e^{-i\sqrt{\lambda}(x,\omega)}dx, 
\qquad \lambda>0, \quad \omega\in\mathbb S^{d-1}.
$$
This map is unitary and diagonalises $H_0$:
$$
(\mathcal F H_0 f)(\lambda;\omega)
=
\lambda(\mathcal F f)(\lambda; \omega),
\quad 
\forall f\in C_0^\infty(\R^d).
$$
Since $\mathbf S$ commutes with $H_0$, the operator $\mathcal F$ also diagonalises $\mathbf S$;
i.e. there exists a family of unitary operators $S(\lambda)$, $\lambda>0$ 
in $L^2(\mathbb S^{d-1})$
such that
$$
(\mathcal F \mathbf S f)(\lambda;\cdot)
=
S(\lambda)f(\lambda;\cdot).
$$
The operator $S(\lambda)$ is called the scattering matrix for 
the pair $H_0$, $H$. 
It is well known that $S(\lambda)$ 
depends continuously on $\lambda>0$ in the operator norm, 
$S(\lambda)-I$ is a compact operator for all $\lambda>0$  and 
$\norm{S(\lambda)-I}\to0$ as $\lambda\to+\infty$.

Fix $\lambda_0>0$ and consider the 
family of unitary operators $\{S(\lambda)\}_{\lambda\in[\lambda_0,\infty]}$,
where $S(\infty)$ is defined as the identity operator. 
By the properties of the scattering matrix, this is a norm continuous 
family, the operator $S(\lambda)-I$ is compact for all $\lambda$ and so 
the spectral flow of this family is well defined.  
Of course, the non-compactness of the interval $[\lambda_0,\infty]$ 
does not cause any problem since $\norm{S(\lambda)-I}\to0$ as $\lambda\to+\infty$. 
We denote 
\begin{equation}
\mu(e^{i\theta};\lambda_0)
=
-\mu(e^{i\theta};\{S(\lambda)\}_{\lambda\in[\lambda_0,\infty]})
\label{d12}
\end{equation}
for all $\theta\in(0,2\pi)$. 
The minus sign here is introduced in order to make the above definition
consistent with the notation of \cite{Push2}.

\begin{theorem}\label{thm.d1}
Let $H_0$ and $H$ be as above; assume \eqref{a19}. 
Then:
\begin{enumerate}[\rm (i)]
\item
for any $\lambda>0$, the index $\Xi(\lambda;H,H_0)$ exists if and only if
$-1\notin\sigma(S(\lambda))$;
\item
for any $\lambda>0$, if  the index $\Xi(\lambda;H,H_0)$ exists then 
the identity 
\begin{equation}
\Xi(\lambda; H,H_0)=-\mu(-1;\lambda)
\label{d8}
\end{equation} 
holds true. 
\end{enumerate}
\end{theorem}
In fact, the set $\calN$ (see \eqref{a18}) in this example  can be alternatively described as 
the set of points $\lambda>0$ where $-1\in\sigma(S(\lambda))$; 
see \eqref{d11} below. 

Suppose that $\lambda>0$ is monotonically increasing
and as  $\lambda$ passes through $\lambda_0$, 
an eigenvalue of $S(\lambda)$ crosses $-1$. 
Formula \eqref{d8} shows that the index function $\Xi(\lambda)=\Xi(\lambda;H,H_0)$ 
has a jump at $\lambda=\lambda_0$, i.e. 
$\Xi(\lambda_0+0)-\Xi(\lambda_0-0)=n$. 
The absolute value $\abs{n}$ of this jump 
equals the multiplicity of the eigenvalue of $S(\lambda)$ which crosses $-1$. 
The value of $n$ is positive if the eigenvalue of $S(\lambda)$ crosses $-1$ in the clockwise 
direction and it is negative for the anti-clockwise direction. 
\begin{remark}
In view of Remark~\ref{rmk.a2}, one can argue that \eqref{d8} has some similarity
to the Birman-Krein formula \cite{BK}
$$
\det S(\lambda)=e^{-2\pi i \xi(\lambda; H,H_0)}.
$$
Indeed, both identities relate some regularisation of
\eqref{a7a} to the spectrum of the scattering matrix. 
This similarity becomes more transparent if the Birman-Krein formula 
is written as 
\begin{equation}
\xi(\lambda;H,H_0)
=
-\frac1{2\pi}\arg\det S(\lambda)
=
-\frac1{2\pi}\sum_n \theta_n(\lambda)
\quad (\hskip -4mm \mod 1),
\label{d8a}
\end{equation}
where $e^{i\theta_n(\lambda)}$ are the eigenvalues of the 
scattering matrix $S(\lambda)$. Informally speaking, 
\eqref{d8} is an integer valued version of \eqref{d8a}.
\end{remark}
\begin{remark}\label{rmk.d2}
Following the proof, one can see that Theorem~\ref{thm.d1}
can be extended to a very general class of pairs of operators $H_0$, $H$ 
such that the a.c. spectrum of $H_0$ coincides with a 
semi-axis and the scattering matrix $S(\lambda)$ 
is continuous in $\lambda$ and $\norm{S(\lambda)-I}\to0$
as $\lambda\to\infty$. 
In fact, 
in \cite{Push2}, the eigenvalue counting function $\mu(e^{i\theta};\lambda)$
was defined and studied in a more general  setting without 
any assumptions on the geometry of the a.c. spectrum of $H_0$. 
The identity \eqref{d8} can also be proven in this case. 
\end{remark}

\subsection{Proof of Theorem~\ref{thm.d1}}\label{sec.d6}

(i)
In \cite{Push4} it is proven that for all 
$\lambda>0$, one has
$$
\sigma_\ess\bigl(E((-\infty,\lambda);H)-E((-\infty,\lambda);H_0)\bigr)
=
[-\alpha(\lambda),\alpha(\lambda)],
\quad 
\alpha(\lambda)=\frac12\norm{S(\lambda)-I}.
$$
Thus, $\Xi(\lambda;H,H_0)$ 
exists if and only if $\alpha(\lambda)<1$. 
Since $S(\lambda)$ is unitary, this means that 
$\Xi(\lambda;H,H_0)$ 
exists if and only if $-1\notin\sigma(S(\lambda))$, 
as required.

(ii)
We use the notation \eqref{a15a}. 
By Theorem~\ref{thm.a1}, it suffices to prove that 
\begin{equation}
\mu(-1;\lambda)
=
\Xi(0;J^{-1}+A_0(\lambda),J^{-1})
\label{d7}
\end{equation}
whenever $-1\notin\sigma(S(\lambda))$. 
In fact, we will prove a more general statement:
\begin{equation}
\mu(e^{i\theta};\lambda)
=
\Xi(0;J^{-1}+A_0(\lambda)+\cot(\theta/2)B_0(\lambda),J^{-1}),
\label{d6}
\end{equation}
whenever $e^{i\theta}\notin\sigma(S(\lambda))$.
The proof of this given below
heavily relies on the results of \cite{Push2}. 
We denote by $F(\lambda,\theta)$ the r.h.s. of \eqref{d6}.

1. 
In \cite[Lemma~5.1]{Push2}, it has been proven that
\begin{equation}
\dim\Ker(S(\lambda)-e^{i\theta}I)
=
\dim\Ker(J^{-1}+A_0(\lambda)+\cot(\theta/2)B_0(\lambda))
\label{d11}
\end{equation}
for all $\lambda>0$ and $\theta\in(0,2\pi)$. 
It follows \cite[Lemma~5.3]{Push2}
that 
\begin{equation}
N(e^{i\theta_1},e^{i\theta_2};S(\lambda))
=
F(\lambda,\theta_1)-F(\lambda,\theta_2),
\label{d7a}
\end{equation}
if $e^{i\theta_1}\notin\sigma(S(\lambda))$ and 
$e^{i\theta_2}\notin\sigma(S(\lambda))$. 

2.  Let $[\lambda_1,\lambda_2]$ be an interval such that
for some $\theta_0\in(0,2\pi)$ and all $\lambda\in[\lambda_1,\lambda_2]$ 
one has $e^{i\theta_0}\notin\sigma(S(\lambda))$. 
Then, by \eqref{d11}, we have 
$$
0\notin\sigma(J^{-1}+A_0(\lambda)+\cot(\theta_0/2)B_0(\lambda))
$$
for all $\lambda\in[\lambda_1,\lambda_2]$.  
From here
by Proposition~\ref{prp.b1}(ii) and Lemma~\ref{lma.b2} of the next section
it follows that $F(\lambda,\theta_0)$
is constant in the interval $\lambda\in[\lambda_1,\lambda_2]$ 
and thus $F(\lambda_1,\theta_0)=F(\lambda_2,\theta_0)$.
From here and \eqref{d7a} we get
$$
N(e^{i\theta},e^{i\theta_0};S(\lambda_2))
-
N(e^{i\theta},e^{i\theta_0};S(\lambda_1))
=
F(\lambda_2,\theta)-F(\lambda_1,\theta).
$$
By the definition \eqref{d5} of the spectral flow, it follows
\begin{equation}
\mu(e^{i\theta};\{S(\lambda)\}_{\lambda\in[\lambda_1,\lambda_2]})
\\
=
F(\lambda_2,\theta)-F(\lambda_1,\theta).
\label{d9}
\end{equation}

3.  Let $[\lambda_1,\lambda_2]\subset(0,\infty)$ be an arbitrary interval.  
According to the definition \eqref{d5a}, we need to split $[\lambda_1,\lambda_2]$ 
into subintervals $\Delta_i$ and add the expressions in the r.h.s of \eqref{d9} 
corresponding to these subintervals. This leads to a telescoping sum,
and so we see that formula
\eqref{d9} extends to an arbitrary interval $[\lambda_1,\lambda_2]\subset(0,\infty)$.

4.  Let us fix $\lambda_1>0$ and $\theta\in(0,2\pi)$ and let $\lambda_2\to\infty$. 
From \eqref{a20} by an argument similar to the one used in the proof of 
Theorem~\ref{thm.a6}, it follows that $F(\lambda_2,\theta)=0$
for all sufficiently large $\lambda_2$. Thus, we obtain 
$$
\mu(e^{i\theta};\{S(\lambda)\}_{\lambda\in[\lambda_1,\infty]})
=
-F(\lambda_1,\theta),
$$
and \eqref{d6} follows. 
\qed

\section{Proof of Theorem~\ref{thm.a1} }\label{sec.b}

\subsection{Stability of index}\label{sec.b1}

Recall the following  statement, see e.g.
\cite[Theorem~VIII.20(i)]{RS1}
and \cite[Theorem~VIII.23(b)]{RS1}:
\begin{proposition}\label{prp.b1}
Let $A_n$ and $A$ be selfadjoint operators and suppose
that $A_n\to A$ as $n\to\infty$ in the norm resolvent sense. 
Then:
\begin{enumerate}[\rm (i)]

\item 
If $f$ is a continuous function on $\R$ with $\lim_{\abs{x}\to\infty}f(x)=0$, 
then $\norm{f(A_n)-f(A)}\to0$ as $n\to\infty$. 

\item 
Let $a,b\in\R$, $a<b$, and suppose that
$a\notin\sigma(A)$, $b\notin\sigma(A)$.
Then
$$
\norm{E((a,b);A_n)-E((a,b);A)}\to0
$$
as $n\to\infty$. 
\end{enumerate}
\end{proposition}

Next, we need a stability theorem for the index of a pair of projections. 
Variants of this statement appeared before, see e.g.
\cite[Theorem~3.12]{GM}.
\begin{lemma}\label{lma.b2}
Let $P,Q$ be a Fredholm pair of orthogonal 
projections in a Hilbert space. 
Let $P_n,Q_n$, $n\geq1$, be orthogonal projections
such that 
\begin{equation}
\norm{(P_n-Q_n)-(P-Q)}\to0
\label{b1}
\end{equation}
as $n\to\infty$. 
Then for all sufficiently large $n$, the pair
$P_n,Q_n$ is Fredholm and 
$$
\iindex(P_n,Q_n)=\iindex(P,Q).
$$
\end{lemma}
\begin{proof}
Since $P,Q$ is a Fredholm pair, there exists $a>0$ such that
$$
\sigma(P-Q)\cap(-1,1)\subset [-1+2a,1-2a].
$$
Then $-1+a$ and $1-a$ are not in the spectrum of $P-Q$ 
and so, by Proposition~\ref{prp.b1}(ii),
\begin{align}
\norm{E((1-a,2);P_n-Q_n)-E((1-a,2);P-Q)}&\to0,
\label{b2}
\\
\norm{E((-2,-1+a);P_n-Q_n)-E((-2,-1+a);P-Q)}&\to0,
\label{b3}
\end{align}
as $n\to\infty$. In particular, $\rank E((1-a,2);P_n-Q_n)$
and $\rank E((-2,-1+a); P_n-Q_n)$ are finite for all 
sufficiently large $n$ and so the pair $P_n,Q_n$ is Fredholm.

Finally, from the definition of index and \eqref{a3} we get
\begin{align*}
\iindex(P,Q)&=
\rank E((1-a,2); P-Q)-\rank E((-2,-1+a);P-Q),
\\
\iindex(P_n,Q_n)&=
\rank E((1-a,2); P_n-Q_n)-\rank E((-2,-1+a);P_n-Q_n)
\end{align*}
and so, applying \eqref{b2}, \eqref{b3}, we get the required statement. 
\end{proof}
In what follows, we will consider families of Fredholm pairs of projections $P_s$, $Q_s$ 
such that the difference $P_s-Q_s$ depends continuously on $s$ in the 
operator norm. Lemma~\ref{lma.b2} ensures that in this situation $\iindex(P_s,Q_s)$ 
is independent of $s$.

\subsection{Existence of $\Xi$}\label{sec.b2}
Assume that $H=H_0+V$ where $V=G^*JG$ satisfies 
assumptions \eqref{a8}.
First we need some notation. 
For $\lambda\in\R$, denote 
\begin{align*}
F_0(\lambda)&=GE((-\infty,\lambda); H_0) 
\bigl(GE((-\infty,\lambda); H_0)\bigr)^*,
\\
F(\lambda)&=GE((-\infty,\lambda); H) 
\bigl(GE((-\infty,\lambda); H)\bigr)^*.
\end{align*}
We note that by \eqref{a8}, \eqref{a9}, the operators
$F_0(\lambda)$, $F(\lambda)$ are compact. 
The existence of $\Xi(\lambda;H,H_0)$ will be derived from 
the following result of \cite{Push4}: 

\begin{proposition}\cite[Theorem~2.6]{Push4}\label{prp.b2}
Assume \eqref{a8}. 
Suppose that for some $\lambda\in\R$, the limits 
$T(\lambda+i0)$, $T_0(\lambda+i0)$ 
and the derivatives 
$\frac{d}{d\lambda}F(\lambda)$,
$\frac{d}{d\lambda}F_0(\lambda)$
exist in the operator norm. 
Then the index $\Xi(\lambda;H,H_0)$ exists if and only if 
$J^{-1}+A_0(\lambda)$ is invertible. 
\end{proposition}

We need two simple lemmas. 
\begin{lemma}\label{lma.b1}
Let $\calM$ be a bounded self-adjoint operator with a bounded inverse 
and let $\calT$ be a compact operator. 
Denote $\calA=\Re \calT$, $\calB=\Im \calT$ and assume that $\calB\geq0$ and 
$\Ker(\calM+\calA)=\{0\}$. Then $\calM+\calT$ has a bounded inverse. 
\end{lemma}
\begin{proof}
Since $\calM$ has a bounded inverse and $\calT$ is compact, it suffices
to prove that $\Ker(\calM+\calT)=\{0\}$. 
Suppose that $(\calM+\calT)f=0$ for some vector $f$. 
Then 
$$
((\calM+\calA)f,f)+i(\calB f,f)=0.
$$
Taking imaginary parts yields $(\calB f,f)=0$. 
Since $\calB\geq0$, it follows that $\calB f=0$. 
Thus, $(\calM+\calA)f=0$ and so $f=0$. 
\end{proof}

\begin{lemma}\label{lma.b3}
Assume \eqref{a8} and \eqref{a15}. 
Then the derivative $\frac{d}{d\lambda}F_0(\lambda)$ exists
in the operator norm for all $\lambda\in\Delta$. 
\end{lemma}
\begin{proof}
From the obvious inequality 
$$
0\leq E(\{\lambda\};H_0)\leq \frac{\varepsilon^2}{(H_0-\lambda I)^2+\varepsilon^2I},
\quad \varepsilon>0,
$$
we get 
\begin{equation}
0\leq GE(\{\lambda\};H_0)\bigl(GE(\{\lambda\};H_0)\bigr)^*
\leq
\varepsilon \Im T_0(\lambda+i\varepsilon), 
\quad \varepsilon>0.
\label{b17}
\end{equation}
By \eqref{b15}, this implies that 
$GE(\{\lambda\};H_0)=0$ for all $\lambda\in\Delta_0$. 
Using this, Stone's formula (see e.g. \cite[Theorem~VII.13]{RS1}) yields 
\begin{equation}
((F_0(b)-F_0(a))f,f)
=
\lim_{\varepsilon\to+0}\frac1\pi\int_a^b
\Im(T_0(\lambda+i\varepsilon)f,f)d\lambda
=
\frac1\pi\int_a^b (B_0(\lambda)f,f)d\lambda
\label{b16}
\end{equation}
for any interval $(a,b)\subset\Delta_0$
and any $f\in\calK$. 
From here and the continuity of $B_0(\lambda)$
we get that $F_0(\lambda)$ is differentiable in $\lambda$
in the operator norm.
\end{proof}

\begin{proof}[Proof of Theorem~\ref{thm.a1}(i) and (ii)]

(i) is a trivial consequence of the fact that the eigenvalues
of $J^{-1}+A_0(\lambda)$ near zero depend continuously on 
$\lambda\in\Delta$. 

(ii) Our aim is to use Proposition~\ref{prp.b2}; we need to check 
that the limits and the derivatives mentioned in the hypothesis of
this proposition exist in the operator norm. 

1. 
The limit $T_0(\lambda+i0)$ exists in the operator norm for all 
$\lambda\in\Delta$; this trivially follows from \eqref{a15}. 
The derivative  $\frac{d}{d\lambda}F_0(\lambda)$ exists
in the operator norm for all $\lambda\in\Delta$ by Lemma~\ref{lma.b3}.

2. Consider $T(\lambda+i0)$ and $\frac{d}{d\lambda}F(\lambda)$. 
Let us fix a closed interval $\Delta_0\subset\Delta\setminus\calN$.
For any $\lambda\in\Delta_0$, we have $\Ker(J^{-1}+A_0(\lambda))=\{0\}$ 
and therefore, by Lemma~\ref{lma.b1}, the operator 
$J^{-1}+T_0(\lambda+i0)$ has a bounded inverse. 

By the identity \eqref{a10b}, we have
\begin{equation}
T(z)=J^{-1}-J^{-1}(J^{-1}+T_0(z))^{-1}J^{-1},
\label{b4}
\end{equation}
where the operator $J^{-1}+T_0(z)$ has a bounded inverse for all 
$\Im z\not=0$. 
Since  $J^{-1}+T_0(\lambda+i0)$ is invertible for all 
$\lambda\in\Delta_0$, we obtain that 
$T(z)$ is uniformly continuous in $z$ in the rectangle
$\Re z\in\Delta_0$, $\Im z\in(0,1)$. 
In particular, the limit $T(\lambda+i0)$ exists in the operator
norm for all $\lambda\in\Delta_0$. 

Now we can apply Lemma~\ref{lma.b3} with $\Delta_0$ instead of 
$\Delta$ and with $T(z)$ instead of $T_0(z)$. 
It follows that the derivative $\frac{d}{d\lambda}F(\lambda)$ exists
in the operator norm for all $\lambda\in\Delta_0$.

3. Now we can 
apply  Proposition~\ref{prp.b2} to any $\lambda\in\Delta_0$, 
and the required statement follows. 
\end{proof}

\subsection{Proof of Theorem~\ref{thm.a1}(iii) and (iv)}

In Sections~\ref{sec.bb} and \ref{sec.c} we prove 
\begin{theorem}\label{thm.b1}
Assume \eqref{a8} and suppose that $T_0(z)$ 
is uniformly continuous in the rectangle 
$\abs{\Re z}<1$, $\Im z\in(0,1)$. 
Assume that $J^{-1}+A_0(0)$ is invertible. 
Then the identity 
\begin{equation}
\Xi(0;H,H_0)=-\Xi(0;J^{-1}+A_0(0),J^{-1})
\label{b12}
\end{equation}
holds true. 
\end{theorem}

This theorem will be proved by 
using the Birman-Schwinger principle (Proposition~\ref{prp.a1}) 
and a certain continuous deformation argument.

Now part 
(iii) of Theorem~\ref{thm.a1} follows directly from Theorem~\ref{thm.b1}. 

Let us prove 
Theorem~\ref{thm.a1}(iv).  
Let us fix a closed interval $\Delta_0\subset\Delta\setminus\calN$.
Since $A_0(\lambda)$ depends continuously on $\lambda\in\Delta$, 
by Proposition~\ref{prp.b1}(ii) the projection 
$E((-\infty,0);J^{-1}+A_0(\lambda))$ depends continuously 
on $\lambda\in\Delta_0$. 
Then by Lemma~\ref{lma.b2}, the index
$\Xi(0;J^{-1}+A_0(\lambda),J^{-1})$ is constant for $\lambda\in\Delta_0$. 
By the identity \eqref{a17},  the index 
$\Xi(\lambda;H,H_0)$ is constant for $\lambda\in\Delta_0$, as required.

\section{Proof of Theorem~\ref{thm.b1} }\label{sec.bb}

\subsection{Notation and preliminaries}\label{sec.bb1}
Throughout the rest of the paper, we assume the hypothesis of Theorem~\ref{thm.b1}.
For a function $\omega\in L^\infty(\R)$, $\omega\geq0$, we denote
$G(\omega)=G\omega(H_0)^{1/2}$. 
Since $\omega(H_0)$ is a bounded operator, we have  by \eqref{a8}
$$
\Dom (H_0-aI)^{1/2}\subset \Dom G(\omega)
\quad\text{ and } \quad
G(\omega) (H_0-aI)^{-1/2} \text{ is compact}
$$
for any $a<\inf \sigma(H_0)$. Thus, we can define
the selfadjoint operator 
$$
H(\omega)=H_0+G(\omega)^* J G(\omega)
$$
as a form sum and the compact operators 
\begin{equation}
\begin{split}
T_0(z;\omega)
&=G(\omega) R_0(z) G(\omega)^*
=G\omega(H_0)R_0(z)G^*,
\\
T(z;\omega)&=G(\omega) (H(\omega)-zI)^{-1} G(\omega)^*. 
\end{split}
\label{b14}
\end{equation}
The definition of $T_0(z;\omega)$ and $T(z;\omega)$ can 
be made more rigorous similarly to \eqref{a10}, \eqref{a11}.  
If the limit $T_0(\lambda+i0;\omega)$ exists, we also denote
$A_0(\lambda;\omega)=\Re T_0(\lambda+i0;\omega)$. 

Let $\chi_\delta$ be the characteristic function of the interval 
$(-\delta,\delta)$ in $\R$, where $\delta\in(0,1)$ will be chosen 
later. For $s\in[0,1]$, we set $\omega_s(x)=1-s\chi_\delta(x)$. 
Let us discuss the existence of the limit $T_0(\lambda+i0;\omega_s)$. 
First note that $R_0(z)(1-\chi_\delta(H_0))$ is analytic in $z$ 
for $\abs{\Re z}<\delta$. It follows that $T_0(z; \omega_1)$ 
is analytic in $z$ for $\abs{\Re z}<\delta$.
Next, writing $\chi_\delta=1-\omega_1$, we get
\begin{equation}
T_0(z;\omega_s)
=
T_0(z)-sT_0(z;\chi_\delta)
=
(1-s)T_0(z)+sT_0(z;\omega_1).
\label{b15}
\end{equation}
By the hypothesis of Theorem~\ref{thm.b1},  it follows that
for any $\delta'<\delta$, the operator   
$T_0(z;\omega_s)$ is uniformly continuous 
in the rectangle $\abs{\Re z}<\delta'$, $\Im z\in(0,1)$ 
in the operator norm. 
In particular, the limit $T_0(\lambda+i0;\omega_s)$ 
exists for all $\lambda\in(-\delta,\delta)$.

\subsection{The strategy of the proof of Theorem~\ref{thm.b1}}\label{sec.bb1a}
Our aim is to show that for all sufficiently small $\delta>0$
and all  $s\in[0,1]$ one has
\begin{equation}
\Xi(0;H(\omega_s),H_0)
=
-\Xi(0;J^{-1}+A_0(0;\omega_s),J^{-1}).
\label{b18}
\end{equation}
Clearly, for $s=0$ this is exactly the 
required identity \eqref{b12}. 
In order to prove \eqref{b18}, we first show that 
if $\delta$ is sufficiently small then the operator 
$J^{-1}+A_0(0;\omega_s)$ is invertible for all 
$s\in[0,1]$. Using this fact, the stability of index and 
Proposition~\ref{prp.b2}, we prove that both sides of 
\eqref{b18} are independent of $s\in[0,1]$. 
Thus it suffices to prove \eqref{b18} for $s=1$. 
Finally, for $s=1$ we derive the identity \eqref{b18} 
from the Birman-Schwinger principle (Proposition~\ref{prp.a1}).

\subsection{The limit $\delta\to0$}\label{sec.bb2}

Let us discuss the choice of $\delta$. 
\begin{lemma}\label{lma.b10}
Assume \eqref{a8} and suppose that $T_0(z)$ 
is uniformly continuous in the rectangle 
$\abs{\Re z}<1$, $\Im z\in(0,1)$. 
Then 
$$
\norm{A_0(0;\chi_\delta)}\to0
\quad\text{ as } \delta\to+0.
$$
\end{lemma}

Using Lemma~\ref{lma.b10}, we will choose $\delta$ 
such that
\begin{equation}
\norm{A_0(0; \chi_\delta)}<\frac12 \norm{(J^{-1}+A_0(0))^{-1}}^{-1}.
\label{b13}
\end{equation}
Then 
\begin{equation}
J^{-1}+A_0(0;\omega_s)=J^{-1}+A_0(0)-sA_0(0;\chi_\delta)
\text{ is invertible for all $s\in[0,1]$.}
\label{b13a}
\end{equation}
This suffices for our construction.

\begin{proof}[Proof of Lemma~\ref{lma.b10}]
1.
From \eqref{b16} we get that 
$\frac{d}{d\lambda}F_0(\lambda)=\frac1\pi B_0(\lambda)$
for any $\lambda\in\Delta$. By the spectral theorem, it follows that 
\begin{equation}
T_0(z;\chi_\delta)
=
\int_{-\delta}^\delta
(\lambda-z)^{-1}dF_0(\lambda)
=
\frac1\pi
\int_{-\delta}^\delta
(\lambda-z)^{-1} B_0(\lambda)d\lambda
\label{c2}
\end{equation}
for all $\Im z>0$.

2. By \eqref{c2}, we have
$$
A_0(0;\chi_\delta)
=\lim_{\varepsilon\to+0}
\frac1\pi
\int_{-\delta}^\delta 
\frac{B_0(\lambda)\lambda}{\lambda^2+\varepsilon^2}d\lambda,
$$
where, by our assumptions, the limit exists in the operator norm.
Next, denote
$$
\mathcal A(\delta_1,\delta_2)=
\frac1\pi \int_{\delta_1}^{\delta_2}
\frac{B_0(\lambda)-B_0(-\lambda)}{\lambda}d\lambda, 
\quad 0<\delta_1<\delta_2<1.
$$
Let us prove that 
\begin{equation}
\lim_{\varepsilon\to+0}
\left\|
\frac1\pi\int_{-\delta}^\delta 
\frac{\lambda B_0(\lambda)}{\lambda^2+\varepsilon^2}d\lambda
-
\mathcal A(\varepsilon,\delta)
\right\|
=0
\label{c20}
\end{equation}
for any $\delta>0$. 
This is a well known  argument, see e.g. \cite[Lemma~VI.1.2]{SteinWeiss}. 
Let 
$$
\varphi(\lambda)=
\begin{cases}
\frac{\lambda}{\lambda^2+1} & \text{ if $\abs{\lambda}<1$,}
\\
\frac{\lambda}{\lambda^2+1}-\frac1\lambda & \text{if $1\leq\abs{\lambda}$,}
\end{cases}
$$
and $\varphi_\varepsilon(\lambda)=\varepsilon^{-1}\varphi(\lambda/\varepsilon)$, 
$\varepsilon>0$. Note that $\varphi$ is odd and $\varphi\in L^1(\R)$. 
We have
\begin{multline}
\int_{-\delta}^\delta 
\frac{\lambda B_0(\lambda)}{\lambda^2+\varepsilon^2}d\lambda
-
\pi\mathcal A(\varepsilon,\delta)
=
\int_\R 
\frac{\lambda B_0(\lambda)\chi_\delta(\lambda)}{\lambda^2+\varepsilon^2}d\lambda
-
\int_{\varepsilon<\abs{\lambda}}
\frac{B_0(\lambda)\chi_\delta(\lambda)}{\lambda}d\lambda
\\
=
\int_\R B_0(\lambda)\chi_\delta(\lambda)\varphi_\varepsilon(\lambda)d\lambda
=
\int_\R B_0(\lambda)\chi_\delta(\lambda)\varphi_\varepsilon(\lambda)d\lambda
-
B_0(0)\int_\R \chi_\delta(\lambda)\varphi_\varepsilon(\lambda)d\lambda
\\
=
\int_\R (B_0(\lambda)-B_0(0))\chi_\delta(\lambda)\varphi_\varepsilon(\lambda)d\lambda.
\label{c3}
\end{multline}
Using the fact that $B_0(\lambda)$ is continuous at $\lambda=0$ in the operator norm, 
by a standard 
argument one checks that the integral 
in the r.h.s. of \eqref{c3} tends to zero in the operator norm as $\varepsilon\to+0$.
This proves \eqref{c20}.

3. 
By \eqref{c20}, the limit
$\lim_{\varepsilon\to+0}\mathcal A(\varepsilon,\delta)$
exists in the operator norm and equals $A_0(0;\chi_\delta)$. 
We can rewrite the last statement as
$$
\lim_{\varepsilon\to+0}
(\mathcal A(\varepsilon,1/2)-\mathcal A(\delta,1/2))
=
A_0(0;\chi_\delta), \quad \delta<1/2.
$$
Now it is clear that
$$
\lim_{\delta\to+0}
A_0(0;\chi_\delta)
=
\lim_{\delta\to+0}
\lim_{\varepsilon\to+0}
(\mathcal A(\varepsilon,1/2)-\mathcal A(\delta,1/2))
=0
$$
in the operator norm, as required. 
\end{proof}

\subsection{The case $s=1$}\label{sec.bb3}

\begin{lemma}\label{thm.b5}
Assume \eqref{a8} and suppose that $T_0(z)$ 
is uniformly continuous in the rectangle 
$\abs{\Re z}<1$, $\Im z\in(0,1)$. 
Assume that $J^{-1}+A_0(0)$ is invertible and 
let $\delta>0$ be chosen as in \eqref{b13}. 
Then the index 
$\Xi(0;H(\omega_1),H_0)$ exists and 
\begin{equation}
\Xi(0;H(\omega_1),H_0)
=
-\Xi(0;J^{-1}+A_0(0;\omega_1),J^{-1}).
\label{b6}
\end{equation}
\end{lemma}
\begin{proof}
1. 
Let $\calH_0=\Ran E(\R\setminus(-\delta,\delta);H_0)$. 
It is easy to see that the subspace $\calH_0$ reduces both $H_0$ and
$H(\omega_1)$ (i.e. both $H_0$ and $H(\omega_1)$ commute 
with $E(\R\setminus(-\delta,\delta);H_0)=\omega_1(H_0)$). 
Along with $H_0$, $H(\omega_1)$, $G(\omega_1)$, 
consider the operators 
$h_0=H_0|_{\calH_0}$, $h=H(\omega_1)|_{\calH_0}$, 
$g=G(\omega_1)|_{\calH_0}$.
We have $(-\delta,\delta)\cap \sigma(h_0)=\varnothing$. 
Since $h=h_0+g^*Jg$ and $g^*Jg$ is $h_0$-form compact, 
we also have 
$(-\delta,\delta)\cap\sigma_\ess(h)=\varnothing$.
Next, let $t_0(z)=g(h_0-zI)^{-1}g^*$.
Note that $t_0(z)=T_0(z;\omega_1)$, $\Im z\not=0$, 
and so 
\begin{equation}
t_0(0)=\Re t_0(0)=A_0(0; \omega_1).
\label{c0}
\end{equation}
By our choice \eqref{b13} of $\delta$, it follows (cf. \eqref{b13a}) that 
the operator $J^{-1}+t_0(0)$ 
is invertible. 
Thus, we can apply Proposition~\ref{prp.a1}
to the pair of operators $h_0$, $h$. 
This yields that $0\notin\sigma(h)$ and 
\begin{equation}
\Xi(0;h,h_0)
=
-\Xi(0;J^{-1}+t_0(0),J^{-1}),
\label{c1}
\end{equation}
where the indices $\Xi$ on both sides exist. 

2. 
Let us show that \eqref{c1} is equivalent to \eqref{b6}. 
By \eqref{c0}, the
r.h.s. of \eqref{c1} coincides with the r.h.s. of \eqref{b6}. 
Consider the l.h.s.
With respect to the orthogonal decomposition
$\calH=\calH_0\oplus\calH_0^\perp$ we have
(here and in what follows $\R_-=(-\infty,0)$):
\begin{align*}
E(\R_-;H_0)
&=
E(\R_-;h_0)\oplus E((-\delta,0);H_0),
\\
E(\R_-;H(\omega_1))
&=
E(\R_-;h)\oplus E((-\delta,0);H_0),
\end{align*}
and therefore
$$
E(\R_-;H(\omega_1))-E(\R_-;H_0)
=
(E(\R_-;h)-E(\R_-;h_0))\oplus 0.
$$
It follows that the index $\Xi(0;H(\omega_1),H_0)$ 
exists if and only if $\Xi(0;h,h_0)$ exists and if 
these indices exist, they coincide. 
Thus, from \eqref{c1} we get that $\Xi(0;H(\omega_1),H_0)$ 
exists and \eqref{b6} holds true. 
\end{proof}

\subsection{The proof of Theorem~\ref{thm.b1}}\label{sec.bb4}

The key element in our proof is 
\begin{theorem}\label{thm.b4}
Assume \eqref{a8} and suppose that $T_0(z)$ 
is uniformly continuous in the rectangle 
$\abs{\Re z}<1$, $\Im z\in(0,1)$. 
Assume that $J^{-1}+A_0(0)$ is invertible and 
let $\delta>0$ be chosen as in \eqref{b13}. 
Then the spectral projections 
$$
E(\R_-;H(\omega_s))
\quad\text{ and }\quad
E(\R_-; J^{-1}+A_0(0;\omega_s))
$$
are continuous in $s\in[0,1]$ in the operator norm. 
\end{theorem}
Theorem~\ref{thm.b4} is proven in Section~\ref{sec.c}.

\begin{proof}[Proof of Theorem~\ref{thm.b1}] 
Let $\delta$ be chosen as in \eqref{b13}. 

1.
By Proposition~\ref{rmk.a1}, the index  
$\Xi(0;J^{-1}+A_0(0;\omega_s),J^{-1})$ exists for all $s$. 
Thus, by Lemma~\ref{lma.b2} and Theorem~\ref{thm.b4}, 
the index $\Xi(0;J^{-1}+A_0(0;\omega_s), J^{-1})$ is independent of 
$s\in[0,1]$.

2. 
Let us prove  that the index $\Xi(0;H(\omega_s),H_0)$ 
exists for any $s\in[0,1]$. 
We will use part (ii) of Theorem~\ref{thm.a1}
(this is not a circular argument: part (ii) has already been proven 
in Section~\ref{sec.b2}). 
Let us apply Theorem~\ref{thm.a1}(ii) with the operators  $H_0$, $H(\omega_s)$,  $G(\omega_s)$ 
instead of $H_0$ $H$, $G$. 
As discussed in Section~\ref{sec.bb1}, for any $\delta'<\delta$
the operator $T_0(z;\omega_s)$ is 
uniformly continuous in $z$ for $\abs{\Re z}<\delta'$, 
$\Im z\in(0,1)$. 
Thus, the hypothesis of Theorem~\ref{thm.a1}  is satisfied
with $\Delta=(-\delta',\delta')$. 
By \eqref{b13a}, we have $0\notin\calN$ 
and so the index $\Xi(0;H(\omega_s),H_0)$ exists for any $s\in[0,1]$. 

3. From the previous step of the proof, using  Lemma~\ref{lma.b2} 
and Theorem~\ref{thm.b4} we obtain that 
$\Xi(0;H(\omega_s), H_0)$ is independent of 
$s\in[0,1]$.

4. Using Lemma~\ref{thm.b5}, we obtain
\begin{multline*}
\Xi(0;H,H_0)
=\Xi(0;H(\omega_0),H_0)
=\Xi(0;H(\omega_1),H_0)
=-\Xi(0;J^{-1}+A_0(0;\omega_1),J^{-1})
\\
=-\Xi(0;J^{-1}+A_0(0;\omega_0),J^{-1})
=-\Xi(0;J^{-1}+A_0(0),J^{-1}),
\end{multline*}
which proves \eqref{b12}.
Of course, this argument also shows that \eqref{b18} holds true for any $s\in[0,1]$. 
\end{proof}

\section{Proof of Theorem~\ref{thm.b4}}\label{sec.c}

\subsection{Estimates for $T(z;\omega_s)$}\label{sec.c2}
We use the notation \eqref{b14}. 
\begin{lemma}\label{lma.c1}
Assume \eqref{a8} and suppose that $T_0(z)$ 
is uniformly continuous in the rectangle 
$\abs{\Re z}<1$, $\Im z\in(0,1)$. 
Assume that $J^{-1}+A_0(0)$ is invertible and 
let $\delta>0$ be chosen as in \eqref{b13}. 
Then for some 
$C>0$ the estimates 
\begin{gather}
\norm{T(it;\omega_s)}\leq C, 
\quad t\in(0,1), \quad s\in[0,1], 
\label{c4}
\\
\norm{T(it;\omega_s)-T(it;\omega_r)}\leq C\abs{s-r}, 
\quad t\in(0,1), \quad s,r\in[0,1], 
\label{c5}
\end{gather}
hold true.
\end{lemma}

\begin{proof}
1. 
Similarly to \eqref{a10b}, we have
$$
(J^{-1}+T_0(z;\omega_s))(J-JT(z;\omega_s)J)
=
(J-JT(z;\omega_s)J)(J^{-1}+T_0(z;\omega_s))
=I
$$
and therefore
\begin{equation}
T(z;\omega_s)
=
J^{-1}-J^{-1}(J^{-1}+T_0(z;\omega_s))^{-1}J^{-1}
\label{c6}
\end{equation}
for all  $s\in[0,1]$ and all $\Im z\not=0$.

2. 
By \eqref{b13a}, 
the operator  $J^{-1}+A_0(0; \omega_s)$
is invertible for all $s\in[0,1]$. 
By Lemma~\ref{lma.b1}, it follows that  
 $J^{-1}+T_0(+i0;\omega_s)$ is also invertible for all 
 $s\in[0,1]$. 
 Since the operator $J^{-1}+T_0(it;\omega_s)$
 is uniformly continuous in $s\in[0,1]$, $t\in(0,1)$ 
 in the operator norm, it follows that the norm of 
 the inverse $(J^{-1}+T_0(it;\omega_s))^{-1}$
 is uniformly bounded for $s\in[0,1]$, $t\in(0,1)$. 
 By \eqref{c6}, we obtain the bound \eqref{c4}. 
 
3. 
Using \eqref{c6}, for any $z\in\C\setminus\R$ we obtain
\begin{multline}
T(z; \omega_s)-T(z; \omega_r)
=
J^{-1}(J^{-1}+T_0(z;\omega_r))^{-1}J^{-1}
-
J^{-1}(J^{-1}+T_0(z;\omega_s))^{-1}J^{-1}
\\
=
J^{-1}(J^{-1}+T_0(z;\omega_s))^{-1}
(T_0(z; \omega_s)-T_0(z; \omega_r))
(J^{-1}+T_0(z;\omega_r))^{-1}J^{-1}
\\
=
(r-s)(I-T(z; \omega_s)J)T_0(z; \chi_\delta)
(I-JT(z; \omega_r)).
\label{c6a}
\end{multline}
Since $T_0(z; \chi_\delta)=T_0(z)-T_0(z; \omega_1)$, 
the limit $T_0(+i0;\chi_\delta)$ exists in the operator norm
and therefore $\norm{T_0(it; \chi_\delta)}$ is uniformly bounded
for $t\in(0,1)$. 
Combining this with \eqref{c6a} and the estimate \eqref{c4}, we obtain 
\eqref{c5}. 
\end{proof}

\subsection{Proof of Theorem~\ref{thm.b4}}\label{sec.c3}

\begin{lemma}\label{lma.c2}
Under the assumptions of Theorem~\ref{thm.b4}, 
for all $s\in[0,1]$ one has
\begin{equation}
\Ker H(\omega_s)=\Ker H_0.
\label{c7}
\end{equation}
\end{lemma}
\begin{proof}
Since $T_0(it; \omega_s)$ is bounded 
uniformly in $t\in(0,1)$,  we obtain, as in \eqref{b17}:
$$
G(\omega_s)E(\{0\};H_0)=0.
$$ 
Thus, for any $f\in\Ker H_0$ we get 
$H(\omega_s)f=H_0f+G(\omega_s)^*JG(\omega_s)f=0$. 
We see that $\Ker H_0\subset \Ker H(\omega_s)$. 
Conversely, using the bound \eqref{c4} in the same 
way we obtain $G(\omega_s)E(\{0\};H(\omega_s))=0$. 
It follows that for any $f\in\Ker H(\omega_s)$ we have
$H_0f=H(\omega_s)f-G(\omega_s)^*JG(\omega_s)f=0$
and so $\Ker H(\omega_s)\subset \Ker H_0$. 
\end{proof}

Let us define the functions $\chi_-$, $\zeta$, $\psi$ as follows:
$$
\chi_-(x)=
\begin{cases}
1,& x<0,\\
1/2, & x=0,\\
0,& x>0,
\end{cases}
\quad
\zeta(x)=
\begin{cases}
\frac1\pi\tan^{-1}(1/x), & x\not=0,\\
0, & x=0
\end{cases}
$$
and $\psi(x)=\chi_-(x)+\zeta(x)$. 
By definition, $\psi\in C(\R)$, 
$\psi(x)\to0$ as $x\to\infty$
and $\psi(x)\to1$ as $x\to-\infty$. 
The key statement
in the proof of Theorem~\ref{thm.b4} is 
\begin{lemma}\label{lma.c3}
Under the assumptions of Theorem~\ref{thm.b4}, 
the operator $\zeta(H(\omega_s))$ depends
continuously on $s\in[0,1]$ in the operator norm. 
\end{lemma}
The proof of Lemma~\ref{lma.c3} is given in Sections~\ref{sec.c4}, \ref{sec.c5}.
Now we are ready to provide
\begin{proof}[Proof of Theorem~\ref{thm.b4}]

1. 
Clearly, $A_0(0; \omega_s)$ is continuous in $s$ in the operator norm. 
By our choice of $\delta$ the operator $J^{-1}+A_0(0;\omega_s)$ 
is invertible for all $s\in[0,1]$.  Thus, the continuity of the projection
$E(\R_-; J^{-1}+A_0(0;\omega_s))$ follows directly from Proposition~\ref{prp.b1}(ii). 

2. Consider the projection $E(\R_-; H(\omega_s))$. 
Using  \eqref{c7}, 
we obtain
$$
E(\R_-;H(\omega_s))
=
\chi_-(H(\omega_s))
+
\frac12 E(\{0\};H(\omega_s))
=
\psi(H(\omega_s))
-
\zeta(H(\omega_s))
+
\frac12 E(\{0\};H_0).
$$
By Lemma~\ref{lma.c3}, it remains to prove that 
$\psi(H(\omega_s))$ depends continuously on $s\in[0,1]$ 
in the operator norm.

3. 
Let us prove that $H(\omega_s)$ is continuous in $s$ in the norm 
resolvent sense. 
For any $z\in\C\setminus\R$, similarly to \eqref{a10a}, 
we have the iterated resolvent identity
\begin{equation}
(H(\omega_s)-zI)^{-1}
-
R_0(z)
=
-\omega_s(H_0)^{1/2}(G R_0(\overline z))^*
(J-JT(z;\omega_s)J)GR_0(z)\omega_s(H_0)^{1/2}.
\label{c5a}
\end{equation}
Clearly, $\omega_s(H_0)^{1/2}$ depends continuously 
on $s$ in the operator norm. 
By \eqref{c6a}, the operator $T(z;\omega_s)$ depends 
continuously on $s$ in the operator norm. 
It follows that 
$(H(\omega_s)-zI)^{-1}$ depends 
continuously 
on $s$ in the operator norm.

4. 
It is easy to see that there exists $a\in\R$ such that 
$a<\inf(\sigma(H(\omega_s)))$
for all $s\in[0,1]$. 
Let $\wt \psi\in C(\R)$ be such that $\wt \psi(x)=\psi(x)$ 
for all $x\geq a$ and $\wt \psi(x)=0$ for $x\leq a-1$. 
Then 
$\psi(H(\omega_s))=\wt \psi(H(\omega_s))$ for all $s$. 
By Proposition~\ref{prp.b1}(i), 
the operator $\wt \psi(H(\omega_s))$ is continuous
in $s$ in the operator norm. This proves the required
statement.  
\end{proof}

\subsection{Proof of Lemma~\ref{lma.c3}}\label{sec.c4}

We will use the following elementary representation for the function 
$\zeta$:
$$
\zeta(x)=\frac1\pi \tan^{-1}(1/x)=\frac1{2\pi}\int_{-1}^1 \frac{dt}{x-it}, 
\quad x\not=0.
$$
Using the resolvent identity \eqref{c5a}, from this 
representation we formally obtain:
\begin{multline}
2\pi\bigl(\zeta(H_0)-\zeta(H(\omega_s))\bigr)
=
\int_{-1}^1\bigl( (H(\omega_s)-it)^{-1}-R_0(it)\bigr)dt
\\
=
\omega_s(H_0)^{1/2}
\int_{-1}^1 (GR_0(-it))^*(J-JT(it;\omega_s)J)GR_0(it)\omega_s(H_0)^{1/2}dt.
\label{c9}
\end{multline}
Of course,
the validity of this formula and 
the convergence of the integral in the r.h.s. have to be rigourously justified; 
this will be done below. 
We note that, by \eqref{c7}, the value $\zeta(0)$ is unimportant; 
the contribution from this value cancels out in the l.h.s. of \eqref{c9}.

Let us denote by $X_+$ and $X_-$ the operators from 
$L^2((-1,1);\calK)$ to $\calH$ defined by 
\begin{equation}
X_\pm f=\int_{-1}^1(GR_0(\mp it))^*f(t)dt,
\label{c10}
\end{equation}
where $f$ belongs to the dense set of functions vanishing
in a neighbourhood of  $t=0$. 
In what follows we prove that $X_\pm$ extend to bounded operators
from $L^2((-1,1);\calK)$ to $\calH$. 

Next, denote by $Y(\omega_s)$ the operator in $L^2((-1,1);\calK)$ 
defined by 
\begin{equation}
(Y(\omega_s)f)(t)=
(J-JT(it;\omega_s)J)f(t), 
\quad t\not=0.
\label{c11}
\end{equation}
Note that $T(-it;\omega_s)=T(it;\omega_s)^*$.
By Lemma~\ref{lma.c1}, the operators $Y(\omega_s)$, are bounded 
for all $s$ and 
\begin{equation}
\norm{Y(\omega_s)-Y(\omega_r)}\leq C\abs{s-r}. 
\label{c12}
\end{equation}
In what follows we prove
\begin{lemma}\label{lma.c5}
\begin{enumerate}[\rm (i)]
\item
The operators $X_\pm$ defined by \eqref{c10}
extend to bounded operators from $L^2((-1,1);\calK)$ to $\calH$. 
 
\item 
The identity 
\begin{equation}
2\pi(\zeta(H_0)-\zeta(H(\omega_s)))
=
\omega_s(H_0)^{1/2}X_+Y(\omega_s)X_-^*\omega_s(H_0)^{1/2}
\label{c13}
\end{equation}
holds true. 
\end{enumerate}
\end{lemma}
Now we can provide
\begin{proof}[Proof of Lemma~\ref{lma.c3}]
Since $\omega_s(H_0)^{1/2}$ depend continuously on $s$ in the 
operator norm, from  \eqref{c12} and \eqref{c13} we immediately 
obtain the required statement. 
\end{proof}

\subsection{Proof of Lemma~\ref{lma.c5}}\label{sec.c5}
(i) We will prove the boundedness of $X_+$; the operator $X_-$ 
can be considered in the same way. 
Let 
$$
\mathcal D=C_0^\infty((-1,1)\setminus\{0\};\calK);
$$ 
clearly, $\mathcal D$ is dense in $L^2((-1,1);\calK)$. 
For $f\in\mathcal D$, using the resolvent identity 
$$
(z_1-z_2)R_0(z_1)R_0(z_2)
=
(R_0(z_1)-R_0(z_2)), 
$$
we obtain
\begin{multline}
\norm{X_+f}^2
=
\int_{-1}^1 dt_1 \int_{-1}^1 dt_2 ((GR_0(-it_2))^*f(t_2),(GR_0(-it_1))^*f(t_1))
\\
=\int_{-1}^1dt_1 \int_{-1}^1 dt_2 \frac{i}{t_1+t_2}((T_0(-it_1)-T_0(it_2))f(t_2),f(t_1)).
\label{c14}
\end{multline}
Thus, we are led to the consideration of the operator 
in $L^2((-1,1);\calK)$ with the integral kernel
$(T_0(-it_1)-T_0(it_2))/(t_1+t_2)$. 
For $f\in\mathcal D$, let us define
$$
(Mf)(t_1)=\mbox{v.p.}\int_{-1}^1\frac{f(t_2)}{t_1+t_2}dt_2.
$$
Up to the change of variables $t\mapsto (-t)$, 
this is  the operator of the Hilbert transform
restricted onto the interval $(-1,1)$. 
Since the Hilbert transform is bounded in $L^2$, 
the operator $M$ is bounded in $L^2((-1,1);\calK)$. 

Next, let $\mathbb T$ be the operator in $L^2((-1,1);\calK)$
given by 
$$
(\mathbb T f)(t)=T_0(it)f(t),
\quad t\not=0.
$$
Since the norm of $T_0(it)$ is uniformly bounded, the operator $\mathbb T$ 
is bounded. 
The r.h.s. of  \eqref{c14} can be rewritten as
\begin{multline*}
\lim_{\varepsilon\to+0}
\biggl(
\iint\limits_{\substack{\abs{t_1}\leq1, \abs{t_2}\leq1 \\ \abs{t_1+t_2}>\varepsilon}}
\frac{i(T_0(-it_1)f(t_2),f(t_1))}{t_1+t_2}dt_1\, dt_2
-
\iint\limits_{\substack{\abs{t_1}\leq1, \abs{t_2}\leq1 \\ \abs{t_1+t_2}>\varepsilon}}
\frac{i(T_0(it_2)f(t_2),f(t_1))}{t_1+t_2}dt_1\, dt_2
\biggr)
\\
=
i(\mathbb T^*Mf,f)-i(M\mathbb Tf,f),
\quad f\in\mathcal D,
\end{multline*}
and therefore $X_+$ extends to a bounded operator. 

(ii)
For any $\varepsilon>0$, let 
$$
\zeta_\varepsilon(x)
=
\frac1{2\pi}\int_{-1}^{-\varepsilon}\frac{dt}{x-it}
+
\frac1{2\pi}\int_{\varepsilon}^1 \frac{dt}{x-it}, 
$$
and let $X_\pm(\varepsilon):L^2((-1,1);\calK)\to \calH$ 
be the operators 
$$
X_\pm(\varepsilon)f
=
\int_{-1}^{-\varepsilon}(GR_0(\mp it))^*f(t)dt
+
\int_{\varepsilon}^1 (GR_0(\mp it))^*f(t)dt.
$$
Since the norm $\norm{GR_0(it)}$ is uniformly bounded for $\abs{t}>\varepsilon$, 
it is clear directly from the definition of  $X_\pm(\varepsilon)$
that these operators  are bounded for each $\varepsilon>0$. 
Applying the resolvent 
identity \eqref{c5a}, by a calculation similar to \eqref{c9} we see that 
\begin{equation}
2\pi(\zeta_\varepsilon(H_0)-\zeta_\varepsilon(H(\omega_s)))
=
\omega_s(H_0)^{1/2}X_+(\varepsilon)
Y(\omega_s)
X_-(\varepsilon)^*\omega_s(H_0)^{1/2}
\label{c15}
\end{equation}
holds true. Let us prove that both sides of \eqref{c15} converge 
weakly to the corresponding sides of \eqref{c13} as $\varepsilon\to+0$. 

Since $\zeta_\varepsilon$ is uniformly bounded and 
$\zeta_\varepsilon(x)\to\zeta(x)$ as $\varepsilon\to+0$ 
for all $x\in\R$ (it is here that the choice of the value $\zeta(0)$
is important) we get that the l.h.s. of \eqref{c15} converges 
weakly to the l.h.s. of \eqref{c13}. 

Next, since $X_+^*$ and $X_-^*$ are bounded by part (i) 
of the Lemma, for any $g\in\calH$ we have 
$$
(X_\pm^*g)(t)=GR_0(\mp it)g, 
\quad t\not=0,
$$
and
$$
\int_{-1}^1\norm{GR_0(it)g}_\calK^2dt<\infty.
$$
It follows that for any $g\in\mathcal H$
$$
\norm{(X_\pm^*(\varepsilon)-X_\pm^*)g}^2
=
\int_{-\varepsilon}^\varepsilon 
\norm{GR_0(it)g}_\calK^2dt\to0
$$
as $\varepsilon\to+0$. 
Thus, $X_\pm^*(\varepsilon)$ converges strongly 
to $X_\pm^*$ as $\varepsilon\to+0$. 
It follows that the r.h.s. of \eqref{c15} converges 
weakly to the r.h.s. of \eqref{c13}. 
This completes the proof.
\qed

\section*{Acknowledgements}
The author is grateful to Serge Richard, Dmitri Yafaev and Nikolai Filonov 
for careful critical reading 
of the manuscript and for offering a number of useful remarks.


\begin{thebibliography}{10}


\bibitem{ADH}
{\sc S.~Alama, P.~A.~Deift, R.~Hempel, }
\emph{Eigenvalue branches of the Schr\"odinger operator $H-\lambda W$ in a gap of $\sigma(H)$.} 
Comm. Math. Phys. \textbf{121} (1989), no.~2, 291--321.

\bibitem{ASS}
{\sc J.~Avron,  R.~Seiler,  B.~Simon,}
\emph{The index of a pair of projections} 
J. Funct. Anal. \textbf{120} (1994), no.~1, 220--237. 

\bibitem{ACS}
{\sc N.~A.~Azamov, A.~L.~Carey, F.~A.~Sukochev},
\emph{The spectral shift function and spectral flow,}
Comm. Math. Phys. {\bf 276} (2007), no.~1, 51--91.

\bibitem{ACDS}
{\sc N.~A.~Azamov, A.~L.~Carey, P.~G.~Dodds, F.~A.~Sukochev},
\emph{Operator integrals, spectral shift and spectral flow,}
Canadian J. of Mathematics, \textbf{61} (2009), no.~2, 241--263. 



\bibitem{Birman}
{\sc M.~Sh.~Birman}, 
\emph{The spectrum of singular boundary problems,}
Amer. Math. Soc. Transl. (2) \textbf{53} (1966),  23--80.


\bibitem{BK}
{\sc M.~Sh.~Birman, M.~G.~Krein,} 
\emph{On the theory of wave operators and
scattering operators}, Soviet Math. Dokl. \textbf{3} (1962),
740--744.


\bibitem{BPR}
{\sc V.~Bruneau, A.~Pushnitski, G.~Raikov}, 
\emph{Spectral Shift Function in Strong Magnetic Fields},
St. Petersburg Math. J. \textbf{16} (2004),  no.~1, 207--238.

\bibitem{DH}
{\sc P.~A.~Deift, R.~Hempel,} 
\emph{On the existence of eigenvalues of the Schr\"odinger operator 
$H-\lambda W$ in a gap of $\sigma(H)$,}
Commun. Math. Phys. \textbf{103} (1986), 461--490. 



\bibitem{GM}
{\sc F.~Gesztesy, K.~Makarov},
\emph{ The $\Xi$ operator and its relation to Krein's spectral shift 
function}, J. Anal. Math. \textbf{81} (2000), 139--183. 

\bibitem{GMN}
{\sc F.~Gesztesy, K.~Makarov, S.~Naboko},
\emph{The spectral shift operator,} in Mathematical 
Results in Quantum Mechanics, J. Dittrich, P. Exner, and M. Tater (eds.), Operator Theory: 
Advances and Applications, Vol. 108, Birkh\"auser, Basel, 1999, 59--90. 

\bibitem{Hempel1}
{\sc R.~Hempel},
\emph{A left-indefinite generalized eigenvalue problem for Schr\"odinger operators,} 
Habilitation thesis, Munich University, 1987.

\bibitem{Hempel2}
{\sc R.~Hempel}, 
\emph{On the asymptotic distribution of the eigenvalue branches of a 
Schr\"odinger operator $H-\lambda W$ in a spectral gap of $H$},
J. Reine Angew. Math. \textbf{399} (1989), 38--59. 

\bibitem{Hempel}
{\sc R.~Hempel},
\emph{Eigenvalues of Schr\"odinger operators in gaps of the essential spectrum --- an overview}, 
Contemporary Mathematics,  \textbf{458}  (2008), Amer. Math. Soc. 



\bibitem{Klaus}
{\sc M.~Klaus}, 
\emph{Some applications of the Birman-Schwinger principle,} 
Helv. Phys. Acta. 
\textbf{55} (1982), 49--68. 


\bibitem{Simon100DM}
{\sc D.~Hundertmark, B.~Simon},
\emph{Eigenvalue bounds in the gaps of Schr\"odinger operators and Jacobi matrices,}
J. Math. Anal. Appl. \textbf{340} (2008), 892--900.

\bibitem{KM}
{\sc V.~Kostrykin, K.~Makarov,} 
\emph{On Krein's example,}
Proc. Amer. Math. Soc. \textbf{136} (2008), no.~6, 2067--2071.


\bibitem{KMS}
{\sc V.~Kostrykin, K.~A.~Makarov,  A.~Skripka}, 
\emph{The Birman-Schwinger principle in von Neumann algebras of finite type,}
J. Funct. Anal. \textbf{247} (2007), no.~2, 492--508. 


\bibitem{Krein}
{\sc M.~G.~Kre\u{\i}n,} 
\emph{On the trace formula in perturbation theory (Russian)},
Mat. Sb. \textbf{33 (75)} (1953), no.~3,  597--626.




\bibitem{Push2}
{\sc  A.~Pushnitski,}
\emph{The spectral shift function and the invariance principle,}
J. Functional Analysis, \textbf{183} (2001), no.~2, 269--320.



\bibitem{Push3}
{\sc A.~Pushnitski,}
\emph{Operator theoretic methods for the eigenvalue counting 
function in spectral gaps,}
Ann. Henri Poincar\'e, \textbf{10} (2009), no.~4, 793--822. 

\bibitem{Push4}
{\sc  A.~Pushnitski,}
\emph{Spectral theory of discontinuous functions of self-adjoint operators: 
essential spectrum,}
to appear in Integral Eq. Oper. Theory.



\bibitem{RS1}
{\sc M.~Reed, B.~Simon,}
\emph{Methods of modern mathematical physics. I. Functional analysis.}
Academic Press, 1972.

\bibitem{RS2}
{\sc M.~Reed, B.~Simon,}
\emph{Methods of modern mathematical physics. II. Fourier Analysis, Self-adjointness.}
Academic Press, 1975.

\bibitem{RS3}
{\sc M.~Reed, B.~Simon,}
\emph{Methods of modern mathematical physics. III. Scattering theory.}
Academic Press, 1979.



\bibitem{Schwinger}
{\sc J.~Schwinger}, 
\emph{On the bound states of a given potential,} 
Proc. Natl. Acad. Sci. USA \textbf{47} (1961), 122--129. 

\bibitem{Sobolev}
{\sc A.~V.~Sobolev,} 
\emph{Efficient bounds for the spectral shift function}, 
Ann. Inst. H. Poincar\'e Phys. Th\'eor. \textbf{58} (1993), no.~1, 55--83. 


\bibitem{SteinWeiss}
{\sc E.~M.~Stein, G.~Weiss, }
\emph{Introduction to Fourier analysis on Euclidean spaces,}
Princeton Mathematical Series, No. 32. 
Princeton University Press, Princeton, N.J., 1971.





\bibitem{Yafaev}
{\sc D.~R.~Yafaev,}
\emph{Mathematical scattering theory. General theory.}
American Mathematical Society, Providence, RI, 1992. 



\end{thebibliography}
\end{document}